\documentclass[american,british,english,linesnumbered,ruled]{article}
\usepackage{ae,aecompl}

\usepackage[T1]{fontenc}
\usepackage[latin9]{inputenc}
\usepackage[a4paper]{geometry}
\usepackage{color}
\usepackage{babel}
\usepackage{float}
\usepackage{textcomp}
\usepackage{mathtools}
\usepackage{algorithm2e}
\usepackage{amsmath}
\usepackage{amsthm}
\usepackage{amssymb}
\usepackage{mathdots}
\usepackage{graphicx}
\usepackage{esint}
\usepackage{multirow}
\usepackage[unicode=true,pdfusetitle,
 bookmarks=true,bookmarksnumbered=false,bookmarksopen=false,
 breaklinks=false,pdfborder={0 0 1},backref=false,colorlinks=true]
 {hyperref}
\hypersetup{
citecolor=green,linkcolor=blue}

\makeatletter
\theoremstyle{plain}
\newtheorem{thm}{\protect\theoremname}[section]
\theoremstyle{plain}

\theoremstyle{definition}
\newtheorem{defn}[thm]{\protect\definitionname}
\theoremstyle{plain}
\newtheorem{prop}[thm]{\protect\propositionname}
\theoremstyle{plain}
\newtheorem{lem}[thm]{\protect\lemmaname}
\theoremstyle{definition}

\theoremstyle{remark}

\theoremstyle{plain}

\usepackage{amsmath,amssymb,amsthm}
\usepackage{babel}
\usepackage{algpseudocode}
\usepackage{url}
\usepackage{graphicx} %REMOVE [demo]
\usepackage{subfigure}
\usepackage{lscape}
\usepackage[colorinlistoftodos]{todonotes} %REMOVE THIS

\algnewcommand\And{\textbf{and}}

\providecommand{\lemmaname}{Lemma}
\providecommand{\remarkname}{Remark}
\providecommand{\theoremname}{Theorem}
\providecommand{\factname}{Fact}
\providecommand{\corollaryname}{Corollary}

\addto\captionsbritish{\renewcommand{\definitionname}{Definition}}
\addto\captionsbritish{\renewcommand{\examplename}{Example}}
\addto\captionsbritish{\renewcommand{\lemmaname}{Lemma}}
\addto\captionsbritish{\renewcommand{\propositionname}{Proposition}}
\addto\captionsbritish{\renewcommand{\remarkname}{Remark}}
\addto\captionsbritish{\renewcommand{\theoremname}{Theorem}}
\addto\captionsbritish{\renewcommand{\factname}{Fact}}
\addto\captionsbritish{\renewcommand{\corollaryname}{Corollary}}

\graphicspath{{}}

\makeatother

  \addto\captionsamerican{\renewcommand{\definitionname}{Definition}}
  \addto\captionsamerican{\renewcommand{\examplename}{Example}}
  \addto\captionsamerican{\renewcommand{\factname}{Fact}}
  \addto\captionsamerican{\renewcommand{\propositionname}{Proposition}}
  \addto\captionsamerican{\renewcommand{\theoremname}{Theorem}}
  \addto\captionsamerican{\renewcommand{\corollaryname}{Corollary}}
  \providecommand{\definitionname}{Definition}
  \providecommand{\examplename}{Example}
  \providecommand{\factname}{Fact}
  \providecommand{\propositionname}{Proposition}
\providecommand{\theoremname}{Theorem}
\begin{document}

%\title{\textbf{Globally Convergent Levenberg\textendash Marquardt Methods 
%for H\"{o}lder Metrically Subregular Mappings}}
\title{Finding Zeros of H\"{o}lder Metrically Subregular Mappings via Globally Convergent Levenberg\textendash Marquardt Methods 
}

\author{Masoud Ahookhosh\thanks{Luxembourg Center for Systems Biomedicine,
University of Luxembourg, Campus Belval, 4362 Esch-sur-Alzette, Luxembourg. Present address:  KU Leuven, Stadius Centre for Dynamical Systems, Signal Processing and Data Analytics, Kasteelpark Arenberg 10, 3001 Leuven.
E-mail: \protect \protect\protect\href{http://masoud.ahookhosh@kuleuven.be}{masoud.ahookhosh@kuleuven.be }}, Ronan M.T. Fleming\thanks{Luxembourg Center for Systems Biomedicine,
University of Luxembourg, Campus Belval, 4362 Esch-sur-Alzette, Luxembourg. Present address:  Leiden University, Leiden Academic Centre for Drug Research, Einsteinweg 55, Leiden, NL 2333 CC, Netherlands.
E-mail: \protect\protect\protect\href{mailto:ronan.mt.fleming@gmail.com}{ronan.mt.fleming@gmail.com}}, Phan T. Vuong\thanks{Luxembourg Center for Systems Biomedicine,
University of Luxembourg, Campus Belval, 4362 Esch-sur-Alzette, Luxembourg. Present address: Institute of Statistics and Mathematical Methods in Economics, Vienna University of Technology, Austria.
E-mail: \protect\protect\protect\href{mailto:vuong.phan@tuwien.ac.at}{vuong.phan@tuwien.ac.at}}\vspace{-2cm}
 }
\maketitle
\begin{abstract}
We present two globally convergent Levenberg\textendash Marquardt methods 
for finding zeros of H\"{o}lder metrically subregular mappings that may have non-isolated zeros. 
The first method unifies the Levenberg\textendash Marquardt 
direction and an Armijo-type line search, while the second incorporates this
direction with a nonmonotone trust-region technique. For both methods,
we prove the global convergence to a first-order stationary point of the
associated merit function. Furthermore, the worst-case global complexity 
of these methods are provided, indicating that an approximate stationary point 
can be computed in at most $\mathcal{O}(\varepsilon^{-2})$ 
function and gradient evaluations, for an accuracy parameter $\varepsilon>0$.
We also study the conditions for the proposed methods to converge to a 
zero of the associated mappings. Computing a moiety conserved steady state for biochemical reaction networks can be cast as 
the problem of finding a zero of a H\"{o}lder metrically subregular mapping. We report encouraging 
numerical results for finding a zero of such mappings derived from real-world biological data, which supports our theoretical foundations.
\end{abstract}

\selectlanguage{british}%
\vspace{2mm}

\selectlanguage{english}%

\section{Introduction}

We consider the problem of finding zeros of the nonlinear mapping $h:\mathbb{R}^{m}\rightarrow\mathbb{R}^{n}$, i.e.,
\begin{equation}
h(x)=0,\quad x\in\mathbb{R}^{m},\label{eq:nonequa}
\end{equation}
where $h$ is continuously differentiable and satisfies the {\it H\"{o}lder metric subregularity} (see bellow). The set of zeros of such mappings is denoted by 
$\Omega$, which is assumed to be nonempty. 

A classical approach for finding a solution of~\eqref{eq:nonequa}
is to search for a minimiser of the nonlinear least-squares problem
\begin{equation}
\min_{x\in\mathbb{R}^{m}}\psi(x),\quad\text{with }\psi:\mathbb{R}^{m}\rightarrow\mathbb{\mathbb{R}}\text{ given by }\psi(x):=\frac{1}{2}\|h(x)\|^{2},\label{eq:meritfunc}
\end{equation}
where $\|\cdot\|$ denotes the Euclidean norm. In order to guarantee the quadratic or superlinear convergence of many Newton-type schemes for solving \eqref{eq:meritfunc}, 
the existence of some constant $\beta>0$ satisfying
\begin{equation}\label{eq:leb}
\beta\,\mathrm{dist}(x,\Omega)\leq\|h(x)\|,\quad\forall x\in\mathbb{B}(x^{*},\mathtt{r})
\end{equation}
is assumed, where $\mathbb{B}(x^{*},\mathtt{r})$ stands for the closed ball centered at $x^{*}$
with radius $\mathtt{r}>0$, cf. \cite{Izmailov2014,alefeld_rate_2001}. Such an inequality is referred as an {\it error bound} 
({\it Lipschitzian error bound} or {\it metric regularity}) condition. The notion of error bound has been very popular
during the last few decades to study the local convergence of optimisation methodologies; however, there are many important mappings where \eqref{eq:leb} is not satisfied , see, e.g., \cite{Ahookhosh_2017, Li2015}. This motivated the authors of \cite{Ahookhosh_2017} to propose a weaker condition so-called the {\it H\"{o}lder metric subregularity} ({\it H\"{o}lderian error bound}), i.e.,
\begin{equation}
\beta\,\mathrm{dist}(x,\Omega)\leq\|h(x)\|^{\delta},\quad\forall x\in\mathbb{B}(x^{*},\mathtt{r}),\label{eq:errorBound}
\end{equation}
for $\delta\in {]0,1]}$ and $\mathtt{r}\in {]0,1[}$. There are many mappings satisfying this condition,
see, e.g., \cite{Ahookhosh_2017, Li2015} and references therein. See also Section \ref{sec.numapp} for a real-world nonlinear system satisfying \eqref{eq:errorBound}, but not \eqref{eq:leb}.

The Levenberg\textendash Marquardt method is a standard technique used to solve~\eqref{eq:nonequa},
where, in the current point $x_k$ and for a positive parameter $\mu_{k}$, the convex subproblem
\begin{equation}\label{eq:def_phi_k}
\min_{d\in\mathbb{R}^{m}}\phi_{k}(d),\quad\text{with }\phi_{k}:\mathbb{R}^{m}\rightarrow\mathbb{\mathbb{R}}
\text{ given by }\phi_{k}(d):=\|\nabla h(x_{k})^{T}d+h(x_{k})\|^{2}+\mu_{k}\|d\|^{2},
\end{equation}
is solved to compute a direction $d_{k}$ in which $\nabla h(x_{k})\in \mathbb{R}^{m\times n}$ is the gradient of $h$ at $x_k$. This requires finding the
unique solution of the linear system
\begin{equation}
\left(\nabla h(x_{k})\nabla h(x_{k})^{T}+\mu_{k}I\right)d_{k}=-\nabla h(x_{k})h(x_{k}),\label{eq:LM_direction}
\end{equation}
where $I\in\mathbb{\mathbb{R}}^{m\times m}$ denotes the identity matrix. Then, the next iteration is be generated by 
$x_{k+1}=x_k+d_k$ and this scheme is continued until a stationary point of \eqref{eq:meritfunc} is found, which may correspond to a zero $h$, when certain conditions are satisfied. 

The choice of the parameter $\mu_{k}$ has substantial impacts on
the global convergence, the local convergence rate, and the computational efficiency of 
Levenberg\textendash Marquardt methods, cf. \cite{Ahookhosh_2017,Kelley1,more1978,wright1985,alefeld_rate_2001}. 
Hence, several ways to specify and to adapt this parameter have been proposed;
see, e.g., \cite{fan_quadratic_2005,fischer2002local,alefeld_rate_2001}. A recently proposed Levenberg\textendash Marquardt 
method by the authors \cite{Ahookhosh_2017} suggests an adaptive parameter $\mu_{k}$ 
based on the order $\delta\in\ ]0,1]$ of the H\"{o}lder metric subregularity~\eqref{eq:errorBound}, i.e.,
\begin{equation}
\mu_{k}=\xi_{k}\|h(x_{k})\|^{\eta}+\omega_{k}\|\nabla h(x_{k})h(x_{k})\|^{\eta},\label{eq:muk}
\end{equation}
where $\eta\in{]0,4\delta[}$, $\xi_{k}\in[\xi_{\min},\xi_{\max}]$ and $\omega_k\in[\omega_{\min},\omega_{\max}]$ 
with $\xi_{\min}+\omega_{\min}>0$. In \cite{Ahookhosh_2017}, this Levenberg\textendash Marquardt method, with \emph{adaptive regularisation} (LM-AR), was presented and its local convergence 
was studied for H\"{o}lder metrically subregular mappings. 

If one assumes that the starting point $x_{0}$ is close enough to a solution $x^{*}$ of 
(\ref{eq:meritfunc}), then the Levenberg\textendash Marquardt method is known to 
be quadratically convergent if $\nabla h(x^{*})$ is nonsingular, in which case it is clearly convergent to a solution 
to~\eqref{eq:nonequa}. In fact, the nonsingularity assumption implies that the solution of the minimisation 
problem~\eqref{eq:meritfunc} must be {\it locally unique}; 
see~\cite{bellavia_strong_2015,kanzow_levenbergmarquardt_2004,alefeld_rate_2001}.
However, assuming local uniqueness of the solution might be restrictive
for many applications since the underlying mappings might have {\it non-isolated zeros}. Therefore, much attention has been devoted to
the study of local convergence of the Levenberg\textendash Marquardt method under local error bounds, 
which enables the solution of mappings with non-isolated zeros; see, e.g., 
\cite{bellavia_strong_2015,fan_quadratic_2005,fischer2002local,alefeld_rate_2001}.
In particular, the local convergence of the Levenberg\textendash Marquardt method was
studied in \cite{Ahookhosh_2017} under the H\"{o}lder metric subregularity condition (\ref{eq:errorBound}).

As is the case in many applications, one cannot provide a sufficiently close starting point $x_0$
to a solution $x^{*}$, and therefore the convergence of the Levenberg\textendash Marquardt method is not 
guaranteed, which decreases the chance of practical applicability. 
To overcome this shortcoming, two {\it globalisation techniques} have been proposed to be combined with 
the Levenberg\textendash Marquardt direction, namely,
{\it line search} and {\it trust-region}; see, e.g.,~\cite{Ahookhosh_2015,Kelley1,Kelley2,ortega_iterative_2000}.
Generally, a line search method finds a descent direction $d_{k}$,
specifies a step-size $\alpha_{k}$, generates the new iteration $x_{k+1}=x_{k}+\alpha_{k}d_{k}$,
and repeats this scheme until a stopping criterion holds. The step-size is usually determined by 
an inexact line search such as Armijo, Wolfe, or Goldstein backtracking schemes; see 
\cite{conn2000,nocedal_numerical_2006-1}. In particular, the Armijo line search usually finds $\alpha_{k}$ 
using a backtracking procedure, which ends up with a step-size satisfying
\begin{equation}\label{eq:armijo}
\psi(x_{k}+\alpha_{k}d_{k})\leq\psi(x_{k})+\sigma\alpha_{k}\nabla\psi(x_{k})^{T}d_{k},
\end{equation}
where $\sigma\in{]0,1[}$. In order to provide an outline for trust-region methods, let us define, firstly, the quadratic function $q_k:\mathbb{R}^m\rightarrow \mathbb{R}$ with
\begin{equation}\label{eq:qk}
q_k(d):=\frac{1}{2}\|\nabla h(x_k)^T d+h(x_k)\|^2.
\end{equation}
Then, a Levenberg\textendash Marquardt trust-region method solves the quadratic subproblem (\ref{eq:def_phi_k}) 
to find a direction $d_k$, computes the ratio of the actual reduction to the predicted reduction 
\begin{equation}\label{eq:rk}
r_k:=\frac{\psi(x_k)-\psi(x_k+d_k)}{q_k(0)-q_k(d_k)},
\end{equation}
and updates the parameter $\mu_k$ using $r_k$. For line search and trust-region methods, 
the global convergence to a first-order stationary point of $\psi$ can be guaranteed,
which results to a monotone sequence of function values, i.e., $\psi(x_{k+1})\leq\psi(x_{k})$.

Regardless of the fact that the monotonicity seems natural for the minimisation goal, it has some
drawbacks. We address two of them here: (i) the monotone method
may lose its efficiency if iterations are trapped at the bottom of
a curved narrow valley, where the monotonicity
forces the iterations to trace the valley floor (causing very short
steps or even an undesired zigzagging); (ii) the Armijo-type line search
can break down for very small step-sizes because of rounding errors, when $\psi(x_{k}+\alpha_k d_{k})\simeq\psi(x_{k})$. 
In this case, the point~$x_{k}$ may still
be far from a stationary point of $\psi$; however, the Armijo rule
cannot be satisfied due to indistinguishability of $\psi(x_{k}+\alpha_k d_{k})$ from $\psi(x_{k})$ 
in floating-point arithmetic. To overcome such limitations, 
the seminal article by Grippo et al.~\cite{Grippo_1986} addressed a variant of the
Armijo rule~\eqref{eq:armijo} by substituting $\psi(x_{k})$ with $\psi_{l(k)}=\max_{0\leq j\leq m(k)}\psi(x_{k-j})$,
where $m(0)=0$ and $0\leq m(k)\leq\min\{m(k-1)+1,M\}$, for some nonnegative integer constant~$M$ for all $k\geq 1$.
This does not guarantee the monotonicity condition $\psi(x_{k+1})\leq\psi(x_{k})$ and therefore called {\it nonmonotone}.
Nonmonotonicity has also been studied for trust-region methods by replacing $\psi(x_k)$ with a nonmonotone term;
cf. \cite{Ahookhosh_Amini_2012}. On the basis of many studies in this area, nonmonotone methods have been recognised 
to be globally convergent and computationally efficient, even for highly nonlinear problems; 
see \cite{Ahookhosh_Amini_2012,Amini2014} and references therein.

%%%%%%%%%%%%%%%%%%%%%%%%%%%%%%%%%%%%%%%%
\subsection{Motivation and contribution}
Our analysis was motivated by the problem of finding moiety conserved steady states of deterministic equations representing the dynamical evolution of molecular species abundance in {\it biochemical 
reaction networks}. This problem can be considered as an application of finding zeros of a mapping $h:\mathbb{R}^m\rightarrow\mathbb{R}^n$, that may not satisfy the local error 
bound~\eqref{eq:leb}. It was previously established~\cite{Ahookhosh_2017} that this 
mapping is H\"older metrically subregular and that the merit function is real analytic using
standard biochemical assumptions; cf.~\cite{artacho_accelerating_2015}. Applying a novel Levenberg\textendash Marquardt algorithm with adaptive regularisation (LM-AR), to this problem, we proved local convergence to a zero of $h$  for all such networks if the sufficiently closeness of a starting point to $x^*$ can be assumed \cite{Ahookhosh_2017}. However, 
providing a starting point close enough to $x^*$ remains as a limitation in practice, as
is the case for all local optimisation methods; see, e.g., \cite{conn2000,Dennis1996}. 

The {\it global convergence} and {\it complexity} of iterative methods has been the subject 
of intense debate within the nonlinear optimisation community over the last few decades. While the global convergence guarantees
the convergence of the iteration sequence generated by a method for any given
starting point $x_0$, the worst-case complexity provides an upper bound
on the number iterations or function evaluations needed to reach a stationary
point of the underlying objective function. These two factors are more 
important if the convexity or structured nonconvexity of the objective function is assumed; 
see, e.g., \cite{attouch2010proximal, bolte2014,NesB,NesG,nesterov2013,yudin1983}. 
In the particular case of solving nonlinear least-squares problems by Levenberg-Marquardt methods,
there are less results about their global convergence and complexity, compared with the volume of literature concerning Newton-type methods; cf. \cite{Ueda2010,Ueda2012}. This motivates our aim to study the global convergence properties and complexity of two Levenberg-Marquardt methods using line search and trust-region techniques.

We analyse the global convergence, and investigate the complexity of,
two Levenberg\textendash Marquardt methods using line search scheme and trust-region
globalisation techniques. For the first method, we use $\mu_k$ defined in (\ref{eq:muk}), solve  
the linear system (\ref{eq:LM_direction}) to specify $d_k$, and combine this direction
with a nonmonotone Armijo-type line search. We also propose a modified version of the Levenberg\textendash Marquardt 
parameter (\ref{eq:muk}), which is lower bounded, and combines the associated direction $d_k$
with a trust-region technique to adapt the Levenberg\textendash Marquardt 
parameter. A global convergence analysis is provided for both methods. Moreover, we demonstrate that, for both methods, a first-order 
stationary point is attained after at most $\mathcal{O}(\varepsilon^{-2})$ iterations or 
function evaluations. We also illustrate some special mappings $h$ where the proposed
methods are convergent to a solution to the nonlinear 
system (\ref{eq:nonequa}). Finally, we demonstrate that the application of these two methods mappings
 derived from real-world biochemical reaction networks, from a diverse set of biological species, shows encouraging 
numerical results in practice. To the best of our knowledge, these two Levenberg\textendash Marquardt methods are
the first methods, globally convergent to a stationary point, for finding zeros of the mapping $h$
arising in the study of biological networks. All algorithms are made
available within the COBRA~Toolbox~v.03~\cite{heirendt_creation_2018}, an open source software package for modelling biochemical reaction networks.

This paper has five sections, besides this
introductory section. Section 2 describes a globally convergent
Levenberg\textendash Marquardt line search method. Section 3 addresses
a globally convergent Levenberg\textendash Marquardt trust-region method,
where in both sections the global convergence and complexity of these methods are analysed.
In Section 4, finding a zero of some specific mappings with the proposed
methods is discussed. Section 5 reports encouraging numerical
results for a mapping appearing in biochemical reaction networks.
Finally, conclusions and area for further research are identified in Section 6.

\vspace{-1mm}
%%%%%%%%%%%%%%%%%%%%%%%%%%%%%%%%%%%%%%%%%%%%%%%%%%%%%%%%%%%
%%%%%%%%%%%%%%%%%%%%%%%%%%%%%%%%%%%%%%%%%%%%%%%%%%%%%%%%%%%
\section{Levenberg\textendash Marquardt line search method}
For the sake of simplicity, we define $H(x_{k}):=\nabla h(x_{k})\nabla h(x_{k})^{T}+\mu_{k}I$.
If $x_k$ is not a stationary point of $\psi$, from positive definiteness of $H(x_{k})$, we obtain
\begin{equation}\label{eq:descentDirection}
\nabla\psi(x_{k})^{T}d_{k}=-d_{k}^{T}H(x_{k})d_{k}<0,
\end{equation}
which guarantees the descent property of $d_{k}$ at $x_{k}$. This motivates us to develop a globally convergent
Levenberg\textendash Marquardt method using (\ref{eq:muk}). More precisely, we shall combine the 
Levenberg\textendash Marquardt direction with a nonmonotone Armijo-type line search using the nonmonotone term
\begin{equation}
D_{k}=\begin{cases}
\psi(x_{0}) & \mathrm{if}\:k=0,\\
(1-\theta_{k-1})\psi(x_{k})+\theta_{k-1}D_{k-1} & \mathrm{if}\:k\geq1,
\end{cases}\label{eq:dk}
\end{equation}
where $\theta_{k-1}\in[\theta_{\min},\theta_{\max}]$ and $0\leq\theta_{\min}\leq\theta_{\max}<1$, cf.~\cite{Ahookhosh_2012,Gu20082158}.

A combination of the direction $d_k$ (given by solving \eqref{eq:LM_direction} using the parameter (\ref{eq:muk}))
with a nonmonotone Armijo-type line search using (\ref{eq:dk}) leads to Algorithm \ref{a.line}.
 
\vspace{4mm}
\begin{algorithm}[H]
\DontPrintSemicolon \KwIn{ $x_{0}\in\mathbb{R}^{m}$, $\eta>0$,
$\overline{\alpha}>0$, $\varepsilon>0$, $\rho,\sigma\in{]0,1[}$,
$\xi_0\in[\xi_{min},\xi_{max}]$, $\omega_0 \in [0,\omega_{max}]$, $\theta_0 \in [\theta_{min},\theta_{max}]$;} %\KwOut{$x_b$,~ $f_{x_b}$;}
\Begin{ $k:=0$;~$\mu_{0}:=\xi_{0}\|h(x_{0})\|^{\eta}+\omega_{0}\|\nabla h(x_{0})h(x_{0})\|^{\eta};$\;
\While {$\|h(x_{k})\|>\varepsilon$ or $\|\nabla\psi(x_{k})\|>\varepsilon$}{
solve the linear system~\eqref{eq:LM_direction} to specify the direction
$d_{k}$;~ $\ell=0$;~$\alpha_{k}=\overline{\alpha}$;\;
\While {$\psi({x}_{k}+{\alpha_k}d_{k})>D_{k}+\sigma\alpha_{k}\nabla\psi(x_{k})^{T}d_{k}$}{
$\ell=\ell+1$;~ $\alpha_{k}=\rho^{\ell}\overline{\alpha}$;\;
} $\ell_{k}=\ell$;~$x_{k+1}=x_{k}+\alpha_k d_{k}$;~ update $\xi_{k}$, $\omega_k$, and $\theta_k$;~\;
update $\mu_{k}$ and $D_{k}$ by~\eqref{eq:muk} and~\eqref{eq:dk},
respectively;\; } } \caption{ \textbf{LMLS} (Levenberg--Marquardt Line Search algorithm)}
\label{a.line}
\end{algorithm}

\medskip
 In order to prove the global convergence of the sequence $\{x_{k}\}$
generated by LMLS to a stationary point of $\psi$, we assume that the next assumptions hold:
\begin{description}
\item [{(A1)}] The mapping $h$ is continuously differentiable and H\"{o}lder
metrically subregular of order $\delta\in{]0,1]}$ at $(x^{*},0)$;
i.e., there exist some constants $\beta>0$ and $\mathtt{r}>0$ such that~\eqref{eq:errorBound} holds;
\item [{(A2)}] The lower level set $\mathcal{L}(x_{0}):=\left\{ x\in\mathbb{R}^{m}\:|\:\psi(x)\leq\psi(x_{0})\right\} $
is bounded;
\item [{(A3)}] $\nabla h$ is Lipschitz continuous, i.e.,
\[
\|\nabla h(x)-\nabla h(y)\|\leq L\|x-y\|,\quad\forall x,y\in \mathbb{R}^m.
\]
\end{description}

In the subsequent proposition, we first 
derive a lower bound for the step-size $\alpha_{k}$ 
and give a bound on the total number of function evaluations needed
until the line search (Line 5 of LMLS) is satisfied.

\begin{prop} \label{eq:boundLk}
Let $\left\{ x_{k}\right\} $ be an infinite
sequence generated by LMLS. Then,
\begin{description}
\item [{(i)}] $x_{k}\in\mathcal{L}(x_{0})$;
\item [{(ii)}] if LMLS does not terminate at $x_k$, then
\begin{equation} \label{eq:boundalphak}
\alpha_{k}\geq\frac{\rho(1-\sigma)\mu^{2}}{\vartheta(L_{0}^{2}+\mu)}:=\widehat{\alpha},
\end{equation}
with
\[
\vartheta:=\frac{1}{2}L_0^2+\frac{1}{2}\overline{\alpha}^2 \rho^{-2} 
L^2 \mu^{-2}L_{0}^2\|h(x_{0})\|^2 + (1+\overline{\alpha}\rho^{-1} L_{0}^{2}\mu^{-1})L \|h(x_{0})\|,
\]
for $\mu\in{]0,\varepsilon^{\eta}[}$. Moreover, the inner loop of LMLS is
terminated in a finite number of steps, denoted by $\ell_{k}$, which
satisfies
\begin{equation}
0\leq \ell_k\leq\frac{\log(\widehat{\alpha})-\log(\overline{\alpha})}{\log(\rho)}.\label{eq:boundpk}
\end{equation}
\end{description}
\end{prop}

\begin{proof}
We prove Assertions (i) and (ii) by induction at the same time. Let us assume $i=1$. Since $D_0=\psi(x_0)$, by the traditional results about
the monotone Armijo line search, we have $\psi(x_1) \leq D_0=\psi(x_0)$. This implies that $x_1\in \mathcal{L}(x_0)$. The proof of 
Assertion (ii) is similar to $i=k$, i.e., we therefore omit it.   

We now assume Assertions (i) and (ii) hold for $i=1,\ldots,k-1$ and prove them for $i=k$. Since $x_{k-1}$ satisfies the line search and $\nabla \psi(x_{k-1})^T d_{k-1}<0$, similar to Lemma 2.3 in \cite{Ahookhosh_2012}, we can show $\psi(x_{k})\leq D_{k-1}$. This and
\begin{equation}\label{eq:Dk_equality}
D_{k}-\psi(x_{k})=\theta_{k-1}(D_{k-1}-\psi(x_{k}))\geq 0,\quad D_{k}-D_{k-1}=(1-\theta_{k-1})(\psi(x_{k})-D_{k-1})\leq 0.
\end{equation}
imply $D_k \leq D_{k-1}$ and
\begin{equation}\label{eq:Dk_psik}
\psi(x_k) \leq D_k. 
\end{equation}
Therefore,
\begin{equation}\label{eq:DkPsi0}
\psi(x_{k})\leq D_{k-1}\leq D_{k-2}\leq\ldots\leq D_{0}=\psi(x_{0}),
\end{equation}
leading to $x_k\in \mathcal{L}(x_0)$, i.e., Assertion (i) holds for $i=k$. 

From (A1) and $x_k\in \mathcal{L}(x_0)$, there exists some constant $L_{0}>0$ such that
\begin{equation}\label{eq:boundj}
\|\nabla h(x_k)\|\leq L_{0},
\end{equation}
which implies
\begin{align*}
\|H(x_{k})\| =\left\|\nabla h(x_k)\nabla h(x_k)^T+\mu_{k}I\right\|\leq
\|\nabla h(x_k)\|^2+\mu_k\leq L_{0}^{2}+\mu_{k},
\end{align*}
leading to
\begin{equation}\label{eq:bound_lambda_min}
\lambda_{\min}\left(H(x_{k})^{-1}\right)=\frac{1}{\lambda_{\max}\left(H(x_{k})\right)}=\frac{1}{\|H(x_{k})\|}\geq\frac{1}{L_{0}^{2}+\mu_{k}}.
\end{equation}
From the definition of $d_k$, we obtain
\begin{equation}\label{eq:dkInequality}
\|d_{k}\|=\|H(x_{k})^{-1}\nabla\psi(x_{k})\|\leq\|H(x_{k})^{-1}\|\|\nabla\psi(x_{k})\|=\frac{1}{\lambda_{\min}(H(x_{k}))}\|\nabla\psi(x_{k})\|\leq\frac{1}{\mu_{k}}\|\nabla\psi(x_{k})\|.
\end{equation}
Since LMLS does not stop at $x_{k}$, it holds $\|h(x_{k})\|>\varepsilon$
and $\|\nabla h(x_{k})h(x_{k})\|>\varepsilon$, which imply
\[
\mu_k=\xi_{k}\|h(x_{k})\|^{\eta}+\omega_k\|\nabla h(x_{k})h(x_{k})\|^{\eta}>\xi_{\min}\varepsilon^{\eta},~~~\forall k\geq 0.
\]
Let us consider a constant $\mu\in{]0,\xi_{\min} \varepsilon^{\eta}[}$, i.e.,
\begin{equation}\label{eq:muklowerbound}
\mu_{k}>\mu,~~~\forall k\geq 0.
\end{equation}
We first derive a lower bound on the step-size $\alpha_{k}$. By \eqref{eq:dkInequality}, \eqref{eq:boundj}~ and \eqref{eq:muklowerbound},
we get
\begin{equation}
\|d_{k}\|\leq\mu_{k}^{-1}\|\nabla\psi(x_{k})\|\leq\mu_{k}^{-1}\|\nabla h(x_{k})\|\|h(x_{k})\|\leq\mu_{k}^{-1}L_{0}\|h(x_{0})\|\leq\mu^{-1}L_{0}\|h(x_{0})\|.\label{eq:dkbound}
\end{equation}
Therefore, for all $\alpha >0 $, we have
\begin{equation}
\|h(x_{k})+\alpha\nabla h(x_{k})^{T}d_{k}\|\leq\|h(x_{k})\|+\alpha\|\nabla h(x_{k})\|\|d_{k}\|\leq(1+\alpha L_{0}^{2}\mu^{-1})\|h(x_{0})\|.\label{eq:inequalityh}
\end{equation}
Further, for all $t\in [0,1]$ and $\alpha >0 $, (A3) and~\eqref{eq:dkbound} yield
\begin{equation}
\|\nabla h(x_{k}+t\alpha d_{k})-\nabla h(x_{k})\|\leq L\alpha\|d_{k}\|\leq \alpha L L_{0} \mu^{-1}\|h(x_{0})\|.\label{eq:inequalityn}
\end{equation}
It follows from \eqref{eq:boundj} that
\begin{align}\label{first_Part}
\frac{1}{2}\|h(x_{k})+\alpha\nabla h(x_{k})^{T}d_{k}\|^{2} \nonumber
&=\frac{1}{2}\|h(x_{k})\|^{2}+\alpha h(x_{k})^{T}\nabla h(x_{k})^{T}d_{k}+\frac{1}{2}\alpha^{2}\|\nabla h(x_{k})^{T}d_{k}\|^{2}\\\nonumber
&\leq \frac{1}{2}\|h(x_{k})\|^{2}+\alpha h(x_{k})^{T}\nabla h(x_{k})^{T}d_{k}+\frac{1}{2}\alpha^{2}\|\nabla h(x_{k})\|^2 \|d_{k}\|^{2}\\\nonumber
&\leq \frac{1}{2}\|h(x_{k})\|^{2}+\alpha h(x_{k})^{T}\nabla h(x_{k})^{T}d_{k}+\frac{1}{2}\alpha^{2}L_0^2 \|d_{k}\|^{2}\\
&=\psi(x_{k})+\alpha\nabla\psi(x_{k})^{T}d_{k}+\frac{1}{2}\alpha^{2}L_0^2 \|d_{k}\|^{2}.
\end{align}
By this inequality, the Taylor expansion of $h(x_{k}+\alpha d_{k})$ around $x_{k}$,
and the Cauchy\textendash Schwarz inequality, for any $\alpha>0$, we come to
\begin{equation}\label{eq:Taylor}
\begin{split}
\psi(x_{k}+\alpha d_{k}) 
& =\frac{1}{2}\left\Vert h(x_{k})+\alpha\nabla h(x_{k})^{T}d_{k}+\int_{0}^{1}\alpha(\nabla h(x_{k}+t\alpha d_{k})-\nabla h(x_{k}))^{T}d_{k}~dt\right\Vert ^{2}\\ 
 & =\frac{1}{2}\|h(x_{k})+\alpha\nabla h(x_{k})^{T}d_{k}\|^{2}+\frac{1}{2}\left\Vert \int_{0}^{1}\alpha(\nabla h(x_{k}+t\alpha d_{k})-\nabla h(x_{k}))^{T}d_{k}~dt\right\Vert ^{2}\\
 &\hspace{0.4cm} +(h(x_{k})+\alpha\nabla h(x_{k})^{T}d_{k})^{T}\int_{0}^{1}\alpha(\nabla h(x_{k}+t\alpha d_{k})-\nabla h(x_{k}))^{T}d_{k}~dt\\
&\leq \psi(x_{k})+\alpha\nabla\psi(x_{k})^{T}d_{k}+\frac{1}{2}\alpha^{2}L_0^2 \|d_{k}\|^{2}\\
&\hspace{0.4cm}+ \frac{1}{2}\left(\int_{0}^{1}\alpha(\|\nabla h(x_{k}+t\alpha d_{k})-\nabla h(x_{k})\|~\|d_{k}\|dt\right)^{2}\\
&\hspace{0.4cm}+\|h(x_{k})+\alpha\nabla h(x_{k})^{T}d_{k}\|\int_{0}^{1}\alpha \|\nabla h(x_{k}+t\alpha d_{k})-\nabla h(x_{k})\|~\|d_{k}\|dt.
\end{split}
\end{equation}
This inequality, \eqref{eq:boundj}, \eqref{eq:inequalityh}, \eqref{eq:inequalityn}, and (A3) suggest
\[
\begin{split}
\psi(x_{k}+\alpha d_{k}) \nonumber
&\leq  \psi(x_{k})+\alpha\nabla\psi(x_{k})^{T}d_{k}\\
&\hspace{4mm}+\left(\frac{1}{2}L_0^2+\frac{1}{2}\alpha^2 L^2 \mu^{-2}L_{0}^2\|h(x_{0})\|^2 + (1+\alpha L_{0}^{2}\mu^{-1})L\|h(x_{0})\| \right)\alpha^{2}\|d_{k}\|^{2}.
\end{split}
\]
From \eqref{eq:Dk_psik}, we come to
\begin{equation}\label{psi_estimate}
\begin{split}
\psi(x_{k}+\alpha d_{k}) \nonumber
&\leq D_k+\alpha\nabla\psi(x_{k})^{T}d_{k}\\
&\hspace{4mm}+\left(\frac{1}{2}L_0^2+\frac{1}{2}\alpha^2 L^2 \mu^{-2}L_{0}^2\|h(x_{0})\|^2 + (1+\alpha L_{0}^{2}\mu^{-1})L\|h(x_{0})\| \right)\alpha^{2}\|d_{k}\|^{2}.
\end{split}
\end{equation}
For $\alpha=\alpha_{k}/\rho\leq \overline{\alpha}/\rho$, we have
\begin{align*}
\frac{1}{2}L_0^2+\frac{1}{2}\alpha^2 L^2 \mu^{-2}L_{0}^2\|h(x_{0})\|^2 + (1+\alpha L_{0}^{2}\mu^{-1})L\|h(x_{0})\| & \leq \,  \frac{1}{2}L_0^2+\frac{1}{2}\overline{\alpha}^2 \rho^{-2} L^2 \mu^{-2}L_{0}^2\|h(x_{0})\|^2 \\
&\hspace{4mm}+ (1+\overline{\alpha}\rho^{-1} L_{0}^{2}\mu^{-1})L\|h(x_{0})\|=:\vartheta,
\end{align*}
which yields
\begin{equation}\label{eq:modArmijo}
\psi(x_{k}+\alpha d_{k})\leq D_{k}+\alpha\nabla\psi(x_{k})^{T}d_{k}+\vartheta \alpha^{2}\|d_{k}\|^{2}.
\end{equation}
For $\alpha=\alpha_{k}/\rho$, the Armijo-type
line search (Line 5 of LMLS) does not hold, i.e.,
\[
\psi(x_{k}+\alpha d_{k})>D_{k}+\sigma\alpha\nabla\psi(x_{k})^{T}d_{k}.
\]
This and the inequality~\eqref{eq:modArmijo} lead to
\[
\vartheta\alpha\|d_{k}\|^{2}\geq(\sigma-1)\nabla\psi(x_{k})^{T}d_{k}.
\]
Substituting $\alpha=\alpha_{k}/\rho$, we have thanks to \eqref{eq:dkInequality} and \eqref{eq:bound_lambda_min} that
\begin{equation}
\begin{split}\vartheta\alpha_{k}\rho^{-1}\mu_{k}^{-2}\|\nabla\psi(x_{k})\|^{2} & \geq\vartheta\alpha_{k}\rho^{-1}\|d_{k}\|^{2}>(\sigma-1)\nabla\psi(x_{k})^{T}d_{k}\\
 & =(1-\sigma)\nabla\psi(x_{k})^{T}H(x_{k})^{-1}\nabla\psi(x_{k})\\
 & \geq(1-\sigma)\lambda_{\min}(H(x_{k})^{-1})\|\nabla\psi(x_{k})\|^{2}\\
 & \geq(1-\sigma)(L_{0}^{2}+\mu_k)^{-1}\|\nabla\psi(x_{k})\|^{2}\\
 & \geq(1-\sigma)(L_{0}^{2}+\mu)^{-1}\|\nabla\psi(x_{k})\|^{2}.
\end{split}
\label{eq:lammin}
\end{equation}
It follows from \eqref{eq:lammin} and~\eqref{eq:muklowerbound}
that~\eqref{eq:boundalphak} is valid.
Using $\alpha_{k}=\rho^{\ell_k}\overline{\alpha}$ and~\eqref{eq:boundalphak},
we end up to
\[
\widehat{\alpha}\leq\rho^{\ell_k}\overline{\alpha}\leq\overline{\alpha},
\]
which proves~\eqref{eq:boundpk}.
\end{proof}

The first main result of this section demonstrates some properties of the sequence $D_k$ and
shows that any accumulation point of the
sequence $\left\{ x_{k}\right\}$ generated by LMLS is either a solution
of~\eqref{eq:nonequa} or a stationary point of $\psi$. 

\begin{thm}
\label{thm:suffdecrease}Let $\left\{ x_{k}\right\} $ be an infinite
sequence generated by LMLS. Then, for all $k\geq 0$,
the following assertions hold:
\begin{description}
\item [{(i)}] $\left\{ D_{k}\right\} $ is convergent and
\begin{equation}
\lim_{k\rightarrow\infty}D_{k}=\lim_{k\rightarrow\infty}\psi(x_{k});\label{eq:limdk}
\end{equation}
\item [{(ii)}] $\nabla\psi(x_{k})^{T}d_{k}\leq-c_{1}\|\nabla\psi(x_{k})\|^{2}$;
\item [{(iii)}] LMLS either stops at finite number of iterations, satisfying
$\|h(x_{k})\|\leq\varepsilon$ or $\|\nabla\psi(x_{k})\| \leq\varepsilon$,
or generates an infinite sequence $\left\{ x_{k}\right\} $ such
that any accumulation point of this sequence is a stationary point
of the merit function $\psi$, i.e.,
\begin{equation}
\lim_{k\rightarrow\infty}\|\nabla\psi(x_{k})\|=0.\label{eq:limngrad}
\end{equation}
\end{description}
\end{thm}

\begin{proof}
From \eqref{eq:Dk_equality}, we have $D_k\leq D_{k-1}$. This, Proposition~\ref{eq:boundLk} (i) and (A2)
imply that $\left\{ D_{k}\right\} $ is convergent.
Further, since $\theta_{k-1}\in[\theta_{\min},\theta_{\max}]$, with $\theta_{\max}\in[\theta_{\min},1[$,
we have $1-\theta_{k-1}\geq1-\theta_{\max}>0$. Taking limits from both sides of $D_{k}-D_{k-1}=(1-\theta_{k-1})(\psi(x_{k})-D_{k-1})$ when $k$
goes to infinity, we deduce~\eqref{eq:limdk}.

It follows from \eqref{eq:bound_lambda_min} and the definition of $d_{k}$ that
\begin{equation}\label{eq:gkdk}
\begin{split}
\nabla\psi(x_{k})^{T}d_{k}  &=-\nabla\psi(x_{k})^{T}H(x_{k})^{-1}\nabla\psi(x_{k})\leq-\lambda_{\min}\left(H(x_{k})^{-1}\right)\|\nabla\psi(x_{k})\|^{2} \\
 &\leq-\frac{1}{L_{0}^{2}+\mu_{k}}\|\nabla\psi(x_{k})\|^{2}.
\end{split}
\end{equation}
By the definition of $\mu_{k}$, \eqref{eq:DkPsi0}, and~\eqref{eq:boundj}, we get
\begin{equation}\label{eq:upperBoundMuk}
\begin{split}\mu_{k} & =\xi_{k}\|h(x_{k})\|^{\eta}+\omega_k\|\nabla h(x_{k})h(x_{k})\|^{\eta}\\
&\leq\xi_{\max}\|h(x_{k})\|^{\eta}+\omega_{\max}\|\nabla h(x_{k})\|^{\eta}\|h(x_{k})\|^{\eta}\\
 & \leq(\xi_{\max}+\omega_{\max}L_{0}^{\eta})\|h(x_{0})\|^{\eta}:=\overline{\mu},
\end{split}
\end{equation}
for all $k\geq0$. This and \eqref{eq:gkdk} yield
\[
\nabla\psi(x_{k})^{T}d_{k}\leq-\left(L_{0}^{2}+\overline{\mu}\right)^{-1}\|\nabla\psi(x_{k})\|^{2};
\]
that is, Assertion~(i) holds with $c_{1}:=\left(L_{0}^{2}+\overline{\mu}\right)^{-1}>0$.

Let us now prove the assertion (iii). If the algorithm stops in a finite number of iterations by either
$\|h(x_{k})\|\leq\varepsilon$ or $\nabla\psi(x_{k})\leq\varepsilon$,
the result is valid. Let us assume that the algorithm generates an
infinite sequence $\left\{ x_{k}\right\} $. For a fixed iteration
$x_{k}$, the stopping criteria of LMLS do not hold, i.e.,
$\|h(x_{k})\|>\varepsilon$ and $\|\nabla h(x_{k})h(x_{k})\|>\varepsilon$.
Therefore, from~\eqref{eq:muklowerbound},  we have
\[
\mu_{k} \geq \mu>0.
\]
It can be deduced from Line 5 of LMLS and Assertion (ii) that
\begin{equation}
D_{k}-\psi(x_{k+1})\geq-\sigma\alpha_{k}\nabla\psi(x_{k})^{T}d_{k}\geq c_{1}\sigma\alpha_{k}\|\nabla\psi(x_{k})\|^{2}\geq c_{1}\sigma\widehat{\alpha}\|\nabla\psi(x_{k})\|^{2}.\label{eq:indkpsik1}
\end{equation}
This and the assertion (ii) yield
\[
\lim_{k\rightarrow\infty}\|\nabla\psi(x_{k})\|=\lim_{k\rightarrow\infty}\|\nabla h(x_{k})h(x_{k})\|=0,
\]
i.e., any accumulation point of $\left\{ x_{k}\right\} $ is a stationary
point of $\psi$.
\end{proof}

%%%%%%%%%%%%%%%%%%%

We continue the analysis of LMLS by providing the worst-case global and evaluation
complexities of LMLS, which are upper bounds on the number of iterations
and merit function evaluations required to get an approximate stationary
point of $\psi$ satisfying $\|\nabla\psi(x)\|\leq\varepsilon$, for
the accuracy parameter $\varepsilon$, respectively. Let us denote
by $N_{i}(\varepsilon)$ and $N_{f}(\varepsilon)$ the total number
of iterations and merit function evaluations of LMLS
required to find and an $\varepsilon$-stationary point of \eqref{eq:meritfunc} .

\begin{thm}
\label{thm:complexity} Let $\left\{ x_{k}\right\} $ be the sequence
generated by LMLS and (A2) and (A3) hold. Then,
\begin{description}
\item [{(i)}] the total number of iterations to guarantee $\|\nabla\psi(x_{k})\|\leq\varepsilon$
is bounded above and
\begin{equation}
N_{i}(\varepsilon)\leq\lceil c_{2}^{-1}\psi(x_{0})\varepsilon^{-2}+1\rceil,\label{eq:nieps}
\end{equation}
with $c_{2}:=c_{1}\sigma\widehat{\alpha}(1-\theta_{\max})$;
\item [{(ii)}] the total number of function evaluations to guarantee $\|\nabla\psi(x_{k})\|\leq\varepsilon$
is bounded above and
\begin{equation}
N_{f}(\varepsilon)\leq\frac{\lceil c_{2}^{-1}\psi(x_{0})\varepsilon^{-2}+1\rceil(\log(\widehat{\alpha})-\log(\overline{\alpha}))}{\log(\rho)}.\label{eq:nfeps}
\end{equation}
\end{description}
\end{thm}

\begin{proof}
To prove Assertion (i), we define
\[
\widehat{k}:=\lceil c_{2}^{-1}\psi(x_{0})\varepsilon^{-2}+1\rceil,
\]
which suggests
\begin{equation}
c_{2}\varepsilon^{2}\widehat{k}=c_{2}\varepsilon^{2}\lceil c_{2}^{-1}\psi(x_{0})\varepsilon^{-2}+1\rceil>\psi(x_{0}).\label{eq:ineqkover0}
\end{equation}
Let us assume by contradiction that $N_{i}(\varepsilon)>\widehat{k}$,
which means that the algorithm does not stop in $\widehat{k}$ iterations.
From Line 5 of LMLS,~\eqref{eq:Dk_equality},
and Theorem~\ref{thm:suffdecrease} (ii), we obtain
\[
\begin{split}D_{k}-D_{k+1} & =(1-\theta_{k})(D_{k}-\psi(x_{k+1}))\geq-\sigma\widehat{\alpha}(1-\theta_{k})\nabla\psi(x_{k})^{T}d_{k}\\
 & \geq-\sigma\widehat{\alpha}(1-\theta_{\max})\nabla\psi(x_{k})^{T}d_{k}\geq c_{1}\sigma\widehat{\alpha}(1-\theta_{\max})\|\nabla\psi(x_{k})\|^{2}\\
 & =c_{2}\|\nabla\psi(x_{k})\|^{2},
\end{split}
\]
leading to
\[
\psi(x_{0})=D_{0}\geq D_{0}-D_{\widehat{k}}=\sum_{i=0}^{\widehat{k}-1}(D_{i}-D_{i+1})\geq c_{2}\sum_{i=0}^{\widehat{k}-1}\|\nabla\psi(x_{i})\|^{2}>c_{2}\varepsilon^{2}\widehat{k},
\]
which contradicts to~\eqref{eq:ineqkover0}. Therefore,~\eqref{eq:nieps}
is valid.

Considering the bound on the number of merit function evaluations
in step $k$ ($\ell_{k}$, given in Proposition \ref{eq:boundLk}),
the following upper bound on the total number of merit function evaluations
can be provided by
\[
\begin{split}
N_{f}(\varepsilon) &\leq\sum_{k=0}^{N_{k}(\varepsilon)-1}\ell_k\leq\sum_{k=0}^{N_{k}(\varepsilon)-1}\frac{\log(\widehat{\alpha})-\log(\overline{\alpha})}{\log(\rho)}\\
&\leq\frac{\lceil c_{2}^{-1}\psi(x_{0})\varepsilon^{-2}+1\rceil(\log(\widehat{\alpha})-\log(\overline{\alpha}))}{\log(\rho)},
\end{split}
\]
giving the results.
\end{proof}

Theorem \ref{thm:complexity} implies that the worst-case global and
evaluation complexities of LMLS to attain the approximate
stationary point of $\psi$ are of the order $\mathcal{O}(\varepsilon^{-2})$,
which is the same as the gradient method; see, e.g., \cite{NesB}. However,
in practice Levenberg-Marquardt methods usually performs much better than 
the gradient method.

Let us compute here the second derivative of $\psi$ at $x$, i.e.,
\begin{equation}\label{eq:secondDer}
\nabla^2 \psi(x) = \nabla h(x)  \nabla h(x)^T+\sum_{i=1}^{m} h_i(x) \nabla^2 h_i(x)=  \nabla h(x)  \nabla h(x)^T+S(x),
\end{equation}
where $S(x):= \sum_{i=1}^{m} h_i(x) \nabla^2 h_i(x)$. Three types of the problem (\ref{eq:nonequa}) are recognised with respect to the magnitude of $\|h(x^*)\|$: (i) if $h(x^*)=0$, the
problem is called {\it zero residual}; (ii) if $\|h(x^*)||$ is small, the problem is called {\it small residual}; and if $\|h(x^*)\|$ is
large, the problem is called {\it large residual}; see, e.g., \cite{Dennis1996}. Under the nonsingularity assumption of $\nabla h(x)$ at the limit point $x^*$ of $\{x_k\}$ and using 
(\ref{eq:secondDer}), we investigate the superlinear convergence of 
$\left\{ x_{k}\right\} $  generated by LMLS for zero residual problems, which is the same as the convergence
rate given for quasi-Newton methods; see \cite{Dennis1977}.  

\begin{thm}\label{thm:covRateLMLS}
Let $\psi:\mathbb{R}^m \rightarrow \mathbb{R}$ be twice continuously differentiable on $\mathcal{L}(x_{0})$,
and $\left\{ x_{k}\right\} $ be the sequence generated by LMLS and (A1)-(A3) hold. If 
$\left\{ x_{k}\right\} $ converges to $x^*$ and $\nabla h(x^*)$ has full rank, then 
\begin{equation} \label{eq:limratiosup}
\lim_{k\rightarrow \infty} \frac{\| \nabla \psi(x_k)+\nabla^2 \psi(x_k) d_k\|}{\|d_k\|}  =0,
\end{equation}
there exists $\overline{k}\geq 0$ such that $\alpha_k=1$ for all $k\geq \overline{k}$, and  $\left\{ x_{k}\right\} $ converges 
to $x^*$ superlinearly.
\end{thm}

\begin{proof}
Since $\nabla h(x^*)$ has full rank, (\ref{eq:limngrad}) implies $h(x^*)=0$. This 
and \eqref{eq:secondDer} yield that $\nabla^2 \psi(x^*)$ is positive definite. Hence, $h(x^*)=0$ leads to
\begin{equation}\label{eq:limmuk11}
\lim_{k\rightarrow \infty} \mu_k\leq \xi_{max} \lim_{k\rightarrow \infty} \|h(x_k)\|^\eta+\omega_{max}\lim_{k\rightarrow \infty} \|\nabla \psi(x_k)\| ^\eta=0.
\end{equation}
From \eqref{eq:secondDer}, we obtain
\[
\nabla \psi(x_k)+\nabla^2 \psi(x_k) d_k= (S(x_k)-\mu_k I) d_k,
\]
which implies
\[
\begin{split}
 \frac{\| \nabla \psi(x_k)+\nabla^2 \psi(x_k) d_k\|}{\|d_k\|} & \leq  \frac{\|S(x_k)-\mu_k I\| \|d_k\|}{\|d_k\|}\\
&\leq \|S(x_k)\|+\mu_k = \|\sum_{i=1}^{m} h_i(x) \nabla^2 h_i(x)\|+\mu_k\\
&\leq \sum_{i=1}^{m} \|h_i(x_k)\| \|\nabla^2 h_i(x)\|+\mu_k.
\end{split}
\]
Since $\psi$ is twice continuously differentiable in the compact set $\mathcal{L}(x_{0})$,  $\|\nabla^2 h_i(x)\|$ ($i=1,\ldots,m$) 
is  bounded. This, the last inequality, and (\ref{eq:limmuk11}) give
\[
\lim_{k\rightarrow \infty} \frac{\| \nabla \psi(x_k)+\nabla^2 \psi(x_k) d_k\|}{\|d_k\|}  \leq 
\lim_{k\rightarrow \infty} \sum_{i=1}^{m} \|h_i(x_k)\| \|\nabla^2 h_i(x)\|+\lim_{k\rightarrow \infty} \mu_k=0,
\]
giving (\ref{eq:limratiosup}).

From  Theorem 6.4 in \cite{Dennis1977} and (\ref{eq:limratiosup}), we have that \eqref{eq:armijo} is valid with $\alpha_k=1$,
for all $k$ sufficiently large. Therefore, the superlinear convergence of $\{x_k\}$ follows from Theorem 3.1 in  
\cite{Dennis1977}.
\end{proof}

%%%%%%%%%%%%%%%%%%%%%%%%%%%%%%%%%%%%%%%%%%%%%%%%%%%%%%%%%%%%
%%%%%%%%%%%%%%%%%%%%%%%%%%%%%%%%%%%%%%%%%%%%%%%%%%%%%%%%%%%%
\section{Levenberg\textendash Marquardt trust-region method}
Trust-region methods are known to be effective for nonconvex optimisation problems (see \cite{conn2000}). 
Therefore, this section concerns with the development of a globally convergent Levenberg\textendash Marquardt method
using a trust-region technique and the investigation on its convergence analysis and complexity.

Let us start with some details of a trust-region globalisation technique that will be coupled with 
the Levenberg\textendash Marquardt direction. We first draw your attention to some 
literature, e.g., \cite{Ahookhosh_Amini_2012,Ahookhosh_2015} and references therein, about the 
efficiency of nonmonotone trust-region methods compared to monotone ones for either optimisation 
or nonlinear systems. This motivates us to develop a nonmonotone Levenberg\textendash Marquardt 
trust-region method for solving systems of nonlinear equations.
To do so, we take advantage of the quadratic function $q_k$ (\ref{eq:qk})
and define the ratio
\begin{equation}\label{eq:rkh}
\widehat{r}_k:=\frac{D_k-\psi(x_k+d_k)}{q_k(0)-q_k(d_k)},
\end{equation}
where the nonmonotone term $D_k$ defined by (\ref{eq:dk}). In this ratio, the nominator is called 
nonmonotone reduction and the denominator is called the predicted reduction. Further, let 
us introduce a new Levenberg\textendash Marquardt parameter that is a modified version of (\ref{eq:muk}), 
i.e.,
\begin{equation} \label{eq:mukh}
\widehat{\mu}_k:=\max \left\{ \mu_{min}, \lambda_k \mu_k \right\},
\end{equation}
where $\mu_k$ is given by (\ref{eq:muk}) with $\eta\in{]0,4\delta[}$, $\mu_{min}>0$, $\xi_{k}\in[\xi_{\min},\xi_{\max}]$, 
$\omega_k\in[\omega_{\min},\omega_{\max}]$ with $\xi_{\min}+\omega_{\min}>0$, 
and $\lambda_k$ is updated by
\[
\lambda_{k+1} := \left\{
\begin{array}{ll}
\rho_1 \lambda_k & \mathrm{if}~ \widehat{r}_k< \upsilon_1,\\
\lambda_k     & \mathrm{if}~ \upsilon_1 \leq \widehat{r}_k <\upsilon_2,\\
\rho_2 \lambda_k & \mathrm{if}~ \widehat{r}_k \geq \upsilon_2,
\end{array}
\right.
\]
in which $0<\rho_2<1<\rho_1$ and $0<\upsilon_1<\upsilon_2<1$ are some constants. A simple comparison between
(\ref{eq:muk}) and (\ref{eq:mukh}) indicates that $\widehat{\mu}_k$ is lower bounded and $\lambda_k$ helps
to have a better control on the Levenberg\textendash Marquardt parameter, which shows its effect on numerical
performance of the method (see Section \ref{sec.numapp} for more details).

In our Levenberg\textendash Marquardt trust-region method,
we first determine $\widehat{\mu}_k$ (\ref{eq:mukh}), specify the direction
$d_k$ by solving the linear system (\ref{eq:LM_direction}), and compute the  
ratio $\widehat{r}_k$ (\ref{eq:rkh}). If $\widehat{r}_k\geq \upsilon_1$, the trial point $d_k$ is accepted, 
i.e., $x_{k+1}=x_k+d_k$; otherwise, the parameter $\lambda_k$ should be increased by setting 
$\lambda_k=\rho_1 \lambda_k$. In the case that $\widehat{r}_k\geq \upsilon_2$, the parameter 
$\lambda_k$ is decreased by setting $\lambda_{k+1}=\rho_2 \lambda_k$. 
The final step will be the evaluation of stopping criteria, which here is either 
$\|h(x_{k+1})\|\leq\varepsilon$ or $\|\nabla\psi(x_{k+1})\|\leq\varepsilon$.
We summarise this scheme in Algorithm \ref{a.trustregion}.

\vspace{-1mm}
\begin{algorithm}
\DontPrintSemicolon \KwIn{ $x_{0}\in\mathbb{R}^{m}$, $\eta>0$,
$\varepsilon>0$, $0<\rho_2<1<\rho_1$,$0<\upsilon_1<\upsilon_2<1$,
$\mu_{min}>0$, $\xi_0\in[\xi_{min},\xi_{max}]$, 
$\omega_0 \in [0,\omega_{max}]$, $\theta_0 \in [\theta_{min},\theta_{max}]$;} %\KwOut{$x_b$,~ $f_{x_b}$;}
\Begin{ $k:=0$;~ $\lambda_0=1$;~$\mu_0:=\max \left\{ \mu_{min}, \lambda_0 \left(\xi_0 \|h(x_0)\|^{\eta}+\omega_0\|\nabla h(x_0)h(x_0)\|^{\eta} \right) \right\};$\;
\While {$\|h(x_{k})\|>\varepsilon$ or $\|\nabla\psi(x_{k})\|>\varepsilon$}{
solve the linear system~\eqref{eq:LM_direction} to specify $d_{k}$; 
compute $\widehat{r}_k$ by (\ref{eq:rkh});~ $p=0$;\;
\While {$\widehat{r}_k<\upsilon_1$}{
$p=p+1$,~$\lambda_k=\rho_1^p\lambda_k$;
solve~\eqref{eq:LM_direction} to specify $d_{k}$;~
compute $\widehat{r}_k$ by (\ref{eq:rkh});\;
} 
\eIf{ $\widehat{r}_k\geq\upsilon_2$}{
$\lambda_{k+1}=\rho_2 \lambda_k$;\;
}{
$\lambda_{k+1}=\lambda_k$;\;
}
$p_k=p$;~$x_{k+1}=x_{k}+d_{k}$;~ update $\xi_{k}$, $\omega_k$, and $\theta_k$;
update $\mu_{k}$ and $D_{k}$ by~\eqref{eq:muk} and~\eqref{eq:dk};\; } 
} \caption{ {\bf LMTR} (Levenberg--Marquardt Trust-Region algorithm)}
\label{a.trustregion}
\end{algorithm}

\vspace{2mm}
In LMTR, the loop starts from Line 5 to Line 7 is called the {\it inner loop} and 
the loop starts from Line 3 to Line 14 is called the {\it outer loop}.

The subsequent proposition points out that the inner loop of LMTR
is terminated after a finite number of steps and provides upper bounds for 
$\widehat{\mu}_k$ and $p_k$.

\begin{prop}\label{prop:boundpk}
Let $\left\{ x_{k}\right\} $ be an infinite
sequence generated by LMTR and (A1)-(A3) holds. Then, 
\begin{description}
\item [{(i)}] $q_k(0)-q_k(d_k)\geq \frac{1}{2(L_0^2+\mu_{min})} \|\nabla \psi(x_k)\|^2$; 
\item [{(ii)}] $q_k(0)-q_k(d_k)\leq (\frac{1}{2} L_0^2+\widehat{\mu}_k) \|d_k\|^2$; 
\item [{(iii)}] the inner loop is
terminated in a finite number of steps. Moreover, if LMTR does not terminate at $x_k$, then
\begin{equation} \label{eq:boundmukh}
\widehat{\mu}_k\leq \tau,
\end{equation}
with
\[
\tau:=\frac{\rho_1}{(1-\upsilon_1)}\left(\frac{1}{2}(1+\upsilon_1)L_0^2+\frac{1}{2} L^2 L_{0}^2\mu_{min}^{-2} \|h(x_{0})\|^2 
 + (1+ L_{0}^{2}\mu_{min}^{-1})L\|h(x_{0})\| \right),
\]
and
\begin{equation} \label{eq:boundlk}
p_k \leq \frac{\log(\tau)-\log(\mu_{min}))}{\log(\rho_1)}.
\end{equation}
\end{description}
\end{prop}

\begin{proof}
By the definition of $q_k$ in \eqref{eq:qk} and \eqref{eq:bound_lambda_min}, we get
\begin{equation}\label{eq:pred1}
\begin{split}
q_k(0)-q_k(d_k) &\geq\frac{1}{2} \|h(x_k)\|^2-\frac{1}{2}\|\nabla h(x_k)^T d_k+h(x_k)\|^2-\frac{1}{2} \widehat{\mu}_k \|d_k\|^2\\
& = -\frac{1}{2} d_k^T H_k d_k - \nabla \psi(x_k)^T d_k = \frac{1}{2} \nabla \psi(x_k)^T H_k^{-1} \nabla \psi(x_k)\\
&\geq \frac{1}{2(L_0^2+\mu_{min})} \|\nabla \psi(x_k)\|^2,
\end{split}
\end{equation}
giving Assertion (i). 

It follows from \eqref{eq:qk} that
\begin{equation}\label{eq:pred2}
\begin{split}
q_k(0)-q_k(d_k) &=\frac{1}{2} \|h(x_k)\|^2-\frac{1}{2}\|\nabla h(x_k)^T d_k+h(x_k)\|^2\\
& = -\frac{1}{2} d_k \nabla h(x_k) \nabla h(x_k)^T d_k - h(x_k)^T \nabla h(x_k)^T d_k\\
& = \frac{1}{2} d_k^T H_k d_k +\frac{1}{2} \widehat{\mu}_k \|d_k\|^2 \\
&\leq \left(\frac{1}{2} L_0^2+\widehat{\mu}_k\right) \|d_k\|^2,
\end{split}
\end{equation}
proving Assertion (ii). 

For the first part of the assertion (iii), we show that the inner loop is terminated after a finite number of steps.
From Assertion (i) and \eqref{eq:dkInequality}, we obtain
\begin{equation}\label{eq:pred3}
\begin{split}
q_k(0)-q_k(d_k) \geq \frac{1}{2(L_0^2+\mu_{min})} \|\nabla \psi(x_k)\|^2 \geq \frac{\widehat{\mu}_k^2}{2(L_0^2+\mu_{min})} \|d_k\|^2.
\end{split}
\end{equation}
By (A3) and \eqref{eq:dkInequality}, for $t\in [0,1]$, we get
\begin{equation}\label{eq:ineq00}
\|\nabla h(x_k+t d_k)-\nabla h(x_k)\| \leq t L \|d_k\| \leq t L \widehat{\mu}_k^{-1} \|\nabla \psi (x_k)\| 
\leq L L_0 \widehat{\mu}_k^{-1} \|h(x_k)\|.
\end{equation}
By the Taylor expansion of $\psi(x_k+d_k)$ around $x_k$, we come to
\begin{equation}\label{eq:ActRedIneq0}
\begin{split}
\psi(x_k+d_k) &= \frac{1}{2} \|h(x_k+d_k)\|^2
=\frac{1}{2}\left\Vert h(x_{k})+\nabla h(x_{k})^{T}d_{k}+\int_{0}^{1}(\nabla h(x_{k}+t d_{k})-\nabla h(x_{k}))^{T}d_{k}~dt\right\Vert ^{2} \\
& =\frac{1}{2}\left\Vert h(x_{k})+\nabla h(x_{k})^{T}d_{k}\right\Vert^2+\frac{1}{2}\left\Vert \int_{0}^{1}(\nabla h(x_{k}+t d_{k})-\nabla h(x_{k}))^{T}d_{k}~dt\right\Vert ^{2}\\
&\hspace{4mm}+(h(x_{k})+\nabla h(x_{k})^{T}d_{k})^{T}\int_{0}^{1}(\nabla h(x_{k}+t d_{k})-\nabla h(x_{k}))^{T}d_{k}~dt.
\end{split}
\end{equation}
From this, \eqref{eq:inequalityh}, and \eqref{eq:ineq00}, it consequently holds 
\begin{equation}\label{eq:ActRedIneq1}
\begin{split}
\left| q_k(d_k)-\psi(x_k+d_k) \right| &= \left| \frac{1}{2}\|\nabla h(x_k)^T d_k+h(x_k)\|^2-\psi(x_k+d_k)\right|\\
& = \frac{1}{2}\left\Vert \int_{0}^{1}(\nabla h(x_{k}+t d_{k})-\nabla h(x_{k}))^{T}d_{k}~dt\right\Vert ^{2}\\
&\hspace{4mm}+(h(x_{k})+\nabla h(x_{k})^{T}d_{k})^{T}\int_{0}^{1}(\nabla h(x_{k}+t d_{k})-\nabla h(x_{k}))^{T}d_{k}~dt\\
& \leq  \frac{1}{2} \left(\int_{0}^{1} \Vert\nabla h(x_{k}+t d_{k})-\nabla h(x_{k})\Vert \Vert d_{k} \Vert~dt\right)^{2}\\
&\hspace{4mm}+ \Vert h(x_{k})+\nabla h(x_{k})^{T}d_{k}\Vert \int_{0}^{1}\Vert \nabla h(x_{k}+t d_{k})-\nabla h(x_{k})\Vert \Vert d_{k}\Vert~dt\\
& \leq \left(\frac{1}{2} L^2 L_0^2 \|h(x_k)\|^2 \widehat{\mu}_k^{-2}+(1+L_0^2 \widehat{\mu}_k^{-1}) \|h(x_k)\| \right) \|d_k\|^2.
\end{split}
\end{equation}
Since $\widehat{r}_{k-1}\geq \upsilon_1$, we have $\psi(x_{k})\leq D_{k-1}$. This and
\begin{equation}\label{eq:DkDk1}
D_{k}-D_{k-1}=(1-\theta_{k-1})(\psi(x_{k})-D_{k-1})\leq 0, ~~~ D_{k}-\psi(x_{k})=\theta_{k-1}(D_{k-1}-\psi(x_{k}))\geq 0
\end{equation}
imply $D_{k}\leq D_{k-1}$ and $\psi(x_k)\leq D_k$, leading to  
$\psi(x_k)\leq D_{k}\leq D_{k-1}\leq \ldots \leq D_0 = \psi(x_0)$, i.e.,
\begin{equation}\label{eq:xkLx0}
x_k \in \mathcal{L}(x_0).
\end{equation}
It can be deduced from this and \eqref{eq:ActRedIneq1} that
\begin{equation}\label{eq:ActRedIneq2}
\begin{split}
q_k(d_k)-\psi(x_k+d_k) &\leq \left(\frac{1}{2} L^2 L_0^2 \|h(x_0)\|^2 \widehat{\mu}_k^{-2}+(1+L_0^2 \widehat{\mu}_k^{-1}) \|h(x_0)\| \right) \|d_k\|^2\\
&= \widehat{\mu}_k^{-2} \left(\widetilde{c}_0+\widetilde{c}_1 \widehat{\mu}_k+\widetilde{c}_2 \widehat{\mu}_k^2\right)
\|d_k\|^2,
\end{split}
\end{equation}
where $\widetilde{c}_0:=\frac{1}{2} L^2 L_0^2 \|h(x_0)\|^2$, $\widetilde{c}_1:=L_0^2\|h(x_0)\|$, and
$\widetilde{c}_2:=\|h(x_0)\|$. 
For sufficiently large $p_k$, we have $\widehat{\mu}_k=\rho_1^{p_k} \lambda_k \mu_k$. This, (\ref{eq:rk}), and \eqref{eq:pred3} yield
\[
\begin{split}
\left\vert r_k-1 \right\vert &= \left\vert \frac{q_k(d_k)-\psi(x_k+d_k)}{q_k(0)-q_k(d_k)} \right\vert \\
&\leq 
\frac{2(L_0^2+\mu_{min}) \left(\widetilde{c}_0+\widetilde{c}_1 \rho_1^{p_k} \lambda_k \mu_k +\widetilde{c}_2 \rho_1^{2p_k} \lambda_k^2 \mu_k^2\right)}{\rho_1^{4p_k} \lambda_k^4 \mu_k^4} \rightarrow 0,~~ \mathrm{as}~ p_k \rightarrow +\infty.
\end{split}
\]
It can be deduced from this and $\psi(x_k)\leq D_k$ that $\widehat{r}_k\geq r_k \geq\upsilon_1$, for sufficiently large $p_k$,
proving the first part of Assertion (iii).

In the second part of Assertion (iii), we provide upper bounds for $\widehat{\mu}_k$ and $p_k$. 
Let us denote by $\overline{d}_k$ the solution of the system 
(\ref{eq:LM_direction}) corresponding to the parameter $\overline{\mu}_k:=\rho_1^{p_k-1}\lambda_k \mu_k$ and set 
$\overline{x}_{k+1}=x_k+\overline{d}_k$. By \eqref{eq:dkInequality} and \eqref{eq:boundj},
we get
\begin{equation}
\|\overline{d}_k\|\leq \overline{\mu}_k^{-1}\|\nabla\psi(x_{k})\|
\leq\overline{\mu}_k^{-1}\|\nabla h(x_{k})\|\|h(x_{k})\|
\leq\overline{\mu}_k^{-1}L_{0}\|h(x_{0})\|
\leq\mu_{min}^{-1}L_{0}\|h(x_{0})\|.\label{eq:dkbound1}
\end{equation}
It follows from this and the triangle inequality that
\begin{equation}
\|h(x_{k})+\nabla h(x_{k})^{T}\overline{d}_k\|\leq\|h(x_{k})\|+\|\nabla h(x_{k})\|\|\overline{d}_k\|
\leq(1+ L_{0}^{2}\mu_{min}^{-1})\|h(x_{0})\|.\label{eq:inequalityh1}
\end{equation}
For all $t\in [0,1]$, (A3) and~\eqref{eq:dkbound1} imply
\begin{equation}
\|\nabla h(x_{k}+t \overline{d}_k)-\nabla h(x_{k})\|\leq L\|\overline{d}_k\|\leq L L_{0} \mu_{min}^{-1}\|h(x_{0})\|.\label{eq:inequalityn1}
\end{equation}
From \eqref{eq:boundj}, (\ref{first_Part}), (\ref{eq:ActRedIneq0}), and (\ref{eq:descentDirection}), we obtain
\[
\begin{split}
\psi(x_{k}+\overline{d}_k) &\leq  \psi(x_{k})+\nabla\psi(x_{k})^{T}\overline{d}_k\\
&\hspace{4mm}+\left(\frac{1}{2}L_0^2+\frac{1}{2} L^2 L_{0}^2 \mu_{min}^{-2}\|h(x_{0})\|^2 + (1+ L_{0}^{2}\mu_{min}^{-1})L\|h(x_{0})\| \right)\|\overline{d}_k\|^{2}.
\end{split}
\]
Following $\psi(x_k) \leq D_k$ and $\nabla \psi(x_k)^T \overline{d}_k = -\overline{d}_k^T H_k \overline{d}_k \leq -\overline{\mu}_k \|\overline{d}_k\|^2$, it can be deduced
\begin{equation} \label{eq:psiUpperBound}
\begin{split}
\psi(x_{k}+\overline{d}_k) 
&\leq D_k-\overline{\mu}_k \|\overline{d}_k\|^2\\
&\hspace{4mm}+\left(\frac{1}{2}L_0^2+\frac{1}{2} L^2 L_{0}^2\mu_{min}^{-2} \|h(x_{0})\|^2 + (1+ L_{0}^{2}\mu_{min}^{-1})L\|h(x_{0})\| \right)\|\overline{d}_k\|^{2}.
\end{split}
\end{equation}
It follows from \eqref{eq:pred2} and the definition $\overline{\mu}_k$ that $\widehat{r}_k<\upsilon_1$ and
\[
D_k-\psi(x_k+\overline{d}_k) < \upsilon_1(q_k(0)-q_k(\overline{d}_k)) \leq \upsilon_1 \left( \frac{1}{2} L_0^2+\overline{\mu}_k\right) \|\overline{d}_k\|^{2}.
\] 
Combining this inequality with that in (\ref{eq:psiUpperBound}) suggest
\[
  \left(\overline{\mu}_k-\frac{1}{2}L_0^2-\frac{1}{2} L^2 L_{0}^2\mu_{min}^{-2} \|h(x_{0})\|^2 
 -(1+ L_{0}^{2}\mu_{min}^{-1})L\|h(x_{0})\| \right)\|\overline{d}_k\|^{2} 
 \leq \upsilon_1 \left( \frac{1}{2} L_0^2+\overline{\mu}_k\right) \|\overline{d}_k\|^{2},
\]
leading to
\[
 \widehat{\mu}_k =\rho_1 \overline{\mu}_k \leq \frac{\rho_1}{1-\upsilon_1}\left(\frac{1}{2}(1+\upsilon_1)L_0^2+\frac{1}{2} L^2 L_{0}^2\mu_{min}^{-2} \|h(x_{0})\|^2 
 + (1+ L_{0}^{2}\mu_{min}^{-1})L\|h(x_{0})\| \right),
 \]
giving (\ref{eq:boundmukh}). Since $\widehat{\mu}_k= \rho_1^{p_k} \lambda_k \mu_k$, taking the logarithm from both sides
of
\[
\tau\geq \rho_1^{p_k} \lambda_k \mu_k \geq \rho_1^{p_k} \mu_{min}, 
\]
implies (\ref{eq:boundlk}), completing the proof.
\end{proof}

We now draw your attention to the global convergence of the sequence $\left\{ x_{k}\right\}$ 
generated by LMTR to a first-order stationary point $x^*$ of $\psi$ satisfying $\nabla \psi(x^*)=0$.
Let us first recall the following result for local convergence of the Levenberg\textendash Marquardt method given in
\cite{Ahookhosh_2017}.

\begin{thm}
\label{thm:global_convergence} Let $\left\{ x_{k}\right\} $ be the
sequence generated by LMTR and (A1)-(A3) hold.
Then, $\left\{ D_{k}\right\} $ is convergent and
\begin{equation}
\lim_{k\rightarrow\infty}D_{k}=\lim_{k\rightarrow\infty}\psi(x_{k}).\label{eq:limdk1}
\end{equation}
Further, the algorithm either stops at finite number of iterations, satisfying
$\|h(x_{k})\|\leq\varepsilon$ or $\nabla\psi(x_{k})\leq\varepsilon$,
or generates an infinite sequence $\left\{ x_{k}\right\} $ such
that any accumulation point of this sequence is a stationary point
of the merit function $\psi$, i.e.,
\begin{equation}
\lim_{k\rightarrow\infty}\|\nabla\psi(x_{k})\|=0.\label{eq:limngrad1}
\end{equation}
\end{thm}

\begin{proof}
From \eqref{eq:DkDk1} and \eqref{eq:xkLx0}, we have $D_k \leq D_{k-1}$ and $x_k \in \mathcal{L}(x_0)$. 
Hence, the sequence $\{D_{k}\}$ is decreasing and bounded below, i.e.,
it is convergent. From $\theta_{k}\in[\theta_{\min},\theta_{\max}]$, 
with $\theta_{\max}\in[\theta_{\min},1[$,
we obtain $1-\theta_{k}\geq1-\theta_{\max}>0$. Taking limits when $k$
goes to infinity from $D_k\leq \psi(x_k)\leq D_{k+1}$ gives \eqref{eq:limdk1}. 

If the algorithm stops in a finite number of iterations by either
$\|h(x_{k})\|\leq\varepsilon$ or $\nabla\psi(x_{k})\leq\varepsilon$,
the result is valid. If the algorithm generates the infinite sequence $\left\{ x_{k}\right\} $,
Proposition \ref{prop:boundpk} (i) yields 
\[
D_k-\psi(x_{k+1}) \geq \upsilon_1 (q(0)-q(d_k)) \geq
\frac{\upsilon_1}{2(L_0^2+\mu_{min})} \|\nabla \psi(x_k)\|^2 \geq 0.
\]
From this and \eqref{eq:limdk1}, we obtain
\[
\lim_{k\rightarrow\infty}\|\nabla\psi(x_{k})\|=\lim_{k\rightarrow\infty}\|\nabla h(x_{k})h(x_{k})\|=0,
\]
i.e., any accumulation point of $\left\{ x_{k}\right\} $ is a stationary
point of $\psi$.
\end{proof}

Let us continue this section by providing global and evaluation complexities of the sequence 
$\left\{ x_{k}\right\}$ generated by LMTR using the results presented in Proposition \ref{prop:boundpk}.

\begin{thm}
\label{thm:complexity1} Let $\left\{ x_{k}\right\} $ be the sequence
generated by LMTR and (A1)-(A3) hold. Then,
\begin{description}
\item [{(i)}] the total number of iterations to guarantee $\|\nabla\psi(x_{k})\|\leq\varepsilon$
is bounded above by
\begin{equation}
N_{i}(\varepsilon)\leq\lceil \overline{c}_3^{-1}\psi(x_{0})\varepsilon^{-2}+1\rceil,\label{eq:nieps1}
\end{equation}
where $\overline{c}_3:=\upsilon_1 (1-\eta_{\max})/(2(L_0^2+\mu_{min}))$;
\item [{(ii)}] the total number of function evaluations to guarantee $\|\nabla\psi(x_{k})\|\leq\varepsilon$
is bounded above by
\begin{equation} \label{eq:nfeps1}
N_{f}(\varepsilon)\leq\lceil \overline{c}_3^{-1}\psi(x_{0})\varepsilon^{-2}+1\rceil
\left(\frac{\log(\tau)-\log(\mu_{min}))}{\log(\rho_1)}\right).
\end{equation}
\end{description}
\end{thm}

\begin{proof}
To prove Assertion (i), we first define
\[
\widetilde{k}:=\lceil \overline{c}_3^{-1}\psi(x_{0})\varepsilon^{-2}+1\rceil,
\]
which is equivalent to
\begin{equation}\label{eq:ineqkover}
\overline{c}_3\varepsilon^{2}\widetilde{k}=\overline{c}_3\varepsilon^{2}\lceil \overline{c}_3^{-1}\psi(x_{0})\varepsilon^{-2}+1\rceil>\psi(x_{0}).
\end{equation}
Let us assume by contradiction that $N_{i}(\varepsilon)>\widetilde{k}$,
which means that LMTR does not stop in $\widetilde{k}$ iterations.
For a successful iteration $k$ of LMTR, it follows from \eqref{eq:Dk_equality} 
and Proposition~\ref{prop:boundpk} (i) that
\[
\begin{split}D_{k}-D_{k+1} & =(1-\theta_{k})(D_{k}-\psi(x_{k+1}))\\
&\geq \upsilon_1 (1-\theta_{k}) (q(0)-q(d_k))\\
&\geq  \frac{\upsilon_1 (1-\eta_{\max})}{2(L_0^2+\mu_{min})} \|\nabla\psi(x_{k})\|^2
= \overline{c}_3 \|\nabla\psi(x_{k})\|^2,
\end{split}
\]
leading to
\[
\psi(x_{0})=D_{0}\geq D_{0}-D_{\widetilde{k}}=\sum_{i=0}^{\widetilde{k}-1}(D_{i}-D_{i+1})\geq \overline{c}_3\sum_{i=0}^{\widetilde{k}-1}\|\nabla\psi(x_{i})\|^{2}>\overline{c}_3\varepsilon^{2}\widetilde{k},
\]
which contradicts to~\eqref{eq:ineqkover}, proving Assertion (i).

From (\ref{eq:boundlk}) and (\ref{eq:nieps1}), we obtain 
\[
\begin{split}
N_{f}(\varepsilon)&\leq\sum_{k=0}^{N_{i}(\varepsilon)-1} p_k
\leq\sum_{k=0}^{N_{i}(\varepsilon)-1} \frac{\log(\tau)-\log(\mu_{min}))}{\log(\rho_1)}\\
&\leq \lceil \overline{c}_3^{-1}\psi(x_{0})\varepsilon^{-2}+1\rceil
\left(\frac{\log(\tau)-\log(\mu_{min}))}{\log(\rho_1)}\right),
\end{split}
\]
giving (\ref{eq:nfeps1}).
\end{proof}

We conclude this section by providing the local convergence rate of LMTR if the corresponding 
sequence $\{x_k\}$ is convergent to a solution of \eqref{eq:nonequa} under the \
L{}ojasiewicz gradient inequality (see \cite{lojasiewicz1963,lojasiewicz1965ensembles}). To this end, the presence
of the subsequent two lemmas are necessary in our local analysis of LMTR.

\begin{lem}\label{lem:rate_convergence} 
\cite[Lemma~1]{artacho_accelerating_2015} Let~$\left\{ s_{k}\right\} $ be a sequence
in~$\mathbb{R}_{+}$ and let~$\zeta,\nu$ be some nonnegative constants.
Suppose that $s_{k}\to0$ and that the sequence satisfies
\begin{equation}\label{eq:ineq_seq}
s_{k}^{\zeta}\leq \nu(s_{k}-s_{k+1}),
\end{equation}
for all $k$ sufficiently large. Then
\begin{enumerate}
\item if~$\zeta=0$, the sequence~$\left\{ s_{k}\right\} $ converges
to~$0$ in a finite number of steps;
\item if~$\zeta\in\left ]0,1\right]$, the sequence~$\left\{ s_{k}\right\} $
converges linearly to~$0$ with rate~$1-\frac{1}{\nu}$;
\item if~$\zeta>1$, there exists~$\varsigma>0$ such that, for all $k$ sufficiently large, 
\[
s_{k}\leq\varsigma k^{-\frac{1}{\zeta-1}}.
\]
\end{enumerate}
\end{lem}

\begin{lem}\label{Svaiter_lemma}
\cite[Theorem 2.5 and Lemma 2.3]{karas2015algebraic} The sequence $\{x_{k}\}$ generated by~LMTR with $p=0$ satisfies
$$ \|d_{k}\|\leq\frac{1}{2\sqrt{\widehat{\mu}_{k}}}\|h(x_{k})\|,$$
and
\begin{equation}\label{h(x_k)_estimate}
\|h(x_{k+1})\|^{2}\leq\|h(x_{k})\|^{2}\\
+d_{k}^{T}\nabla h(x_{k})h(x_{k})+\|d_{k}\|^{2}\left[\frac{L^{2}}{4}\|d_{k}\|^{2}+L\|h(x_{k})\|-\widehat{\mu}_{k}\right].
\end{equation}
\end{lem}

Let us describe  now the \emph{\L{}ojasiewicz gradient inequality} in the following definition.
\begin{defn}
Let $\psi:U\to\mathbb{R}$ be a  function defined on an 
open set $U\subseteq\mathbb{R}^{m}$, and assume that the set of zeros
$\Omega:=\left\{ x\in\mathbb{R}^{m},\,\psi(x)=0\right\} $ is nonempty. 
The function $\psi$ is said to satisfy the \L{}ojasiewicz gradient 
inequality if for any critical point~$\overline{x}$, there exist constants  
$\kappa>0,\varepsilon>0$ and $\theta \in[0,1[$  such that
\begin{equation}\label{eq:Lojasiewicz_Gradient_Inequality}
|\psi(x)-\psi(\overline{x})|^{\theta}\leq \kappa \|\nabla \psi(x)\|,\quad\forall x\in\mathbb{B}(\overline{x},\varepsilon).
\end{equation}
\end{defn} 
This inequality is valid for a large class of functions such as analytic, subanalytic, 
and semialgebraic functions, cf. \cite{lojasiewicz1963,lojasiewicz1965ensembles,kurdyka_1998}. 
See Section \ref{sec.numapp} for a mapping with a real analytic merit function, where finding 
zeros of this mapping is the main motivation of this study. Here, we further assume that

\begin{description}
\item [(A4)] the merit function $\psi$ satisfies the \L{}ojasiewicz gradient inequality \eqref{eq:Lojasiewicz_Gradient_Inequality}.
\end{description}

The next theorem is the third main result of this section, which provides the convergence of the sequences 
$\left\{ \mathrm{dist}(x_{k},\Omega)\right\} $ and $\left\{ \psi(x_k)\right\} $ 
to $0$ if an accumulation point $x^*$ of $\{x_k\}$ is a solution of the 
nonlinear system \eqref{eq:nonequa}.

\begin{thm}
\label{thm:convergence_of_dist(x,Omega)_general} Suppose that (A4) holds and assume that the sequence~$\{x_k\}$ generated by~LMTR is convergent to a solution $x^*$ of the nonlinear system \eqref{eq:nonequa}. Then, 
\begin{enumerate}
\item for sufficiently large $k$, it holds $x_{k+1}=x_k+d_k$; 
\item there exist constants $s>0$, $\overline{s}>0$, and $k'\in\mathbb{N}$ such that, for $x_{k'}\in\mathbb{B}(x^{*},s)$, 
\[
\{x_k\}_{k\geq k'}\subset\mathbb{B}(x^*,\overline{s}), \quad \{\psi(x_{k})\}\rightarrow 0, \quad \{\mathrm{dist}(x_{k},\Omega)\}, \quad
\mathrm{as}~ k\rightarrow \infty;
\]
\item if $\theta=0$, the sequences~$\{\psi(x_{k})\}$ and $\{\mathrm{dist}(x_{k},\Omega)\}$ converge to~$0$  in a finite number of steps;
\item if $\theta\in\left ]0,\frac{1}{2}\right]$, the sequences~$\{\psi(x_{k})\}$ and $\{\mathrm{dist}(x_{k},\Omega)\}$ converge linearly to~$0$;
\item if $\theta \in\left ]\frac{1}{2},1\right[$, there exist some positive constants $\varsigma_1$ and $\varsigma_2$ such that, for all large~$k$,
\begin{gather*}
\psi(x_{k})\leq \varsigma_1 k^{-\frac{1}{2\theta-1}}\quad\text{and}\quad
\mathrm{dist}(x_{k},\Omega)\leq \varsigma_2 k^{-\frac{\delta}{2(2\theta-1)}}.
\end{gather*}
\end{enumerate}
\end{thm}

\begin{proof}
Since $x^{*}$ is an accumulation point of $\left\{ x_{k}\right\} $ and  a solution of 
the nonlinear system~\eqref{eq:nonequa}, it can be deduced
\begin{equation}\label{eq:limnormdk1}
\lim_{k\rightarrow\infty} \|h(x_{k})\| = 0. 
\end{equation}
From this, Proposition \ref{prop:boundpk} (i), Proposition \ref{eq:boundLk} (ii),
(\ref{eq:ActRedIneq1}), $\widehat{\mu}_k \geq \mu_{min}$, and \eqref{eq:limnormdk1}, we obtain
\[
\begin{split}
\left\vert r_k-1 \right\vert &= \left\vert \frac{q_k(d_k)-\psi(x_k+d_k)}{q_k(0)-q_k(d_k)} \right\vert\\
& \leq 
\frac{2(L_0^2+\mu_{min})\left(\frac{1}{2} L^2 L_0^2 \|h(x_k)\|^2 \widehat{\mu}_k^{-2}+(1+L_0^2 \widehat{\mu}_k^{-1}) \|h(x_k)\| \right) \|d_k\|^2}{\|\nabla \psi(x_k)\|^2}\\ 
& \leq 
\frac{2(L_0^2+\mu_{min})\left(\frac{1}{2} L^2 L_0^2 \|h(x_k)\|^2 +(\widehat{\mu}_k^2+L_0^2 \widehat{\mu}_k) \|h(x_k)\| \right)}{\widehat{\mu}_k^4}\\
&\leq \frac{2(L_0^2+\mu_{min})\left(\frac{1}{2} L^2 L_0^2 \|h(x_k)\|^2 +(\mu_{min}^2+L_0^2 \mu_{min}) \|h(x_k)\| \right)}{\mu_{min}^4}
\rightarrow 0,~~ \mathrm{as}~ k \rightarrow +\infty, 
\end{split}
\]
which implies that there exists a $\overline{k}_0 \in \mathbb{N}$ such that $r_k\geq \upsilon_1$. Hence, for all $k \geq \overline{k}_0$,
it follows from $D_k\geq \psi(x_k)$ that
\[
\widehat{r}_k=\frac{D_k-\psi(x_k+d_k)}{q_k(0)-q_k(d_k)}\geq
\frac{\psi(x_k)-\psi(x_k+d_k)}{q_k(0)-q_k(d_k)}=r_k \geq \upsilon_1,
\] 
which means that $x_{k+1}=x_k+d_k$ with $p_k=0$ that justifies Assertion (i).

We divide the proof of Assertion (ii) into three parts. First, we will provide the values of $s$ and $\overline{s}$.
Let us set $\varepsilon>0$ and $\kappa>0$ such that~\eqref{eq:Lojasiewicz_Gradient_Inequality} holds and 
let $\overline{s}:=\min\{\mathtt{r},\varepsilon\}>0$.
By the definition of~$\widehat{\mu}_k$, (A2), and \eqref{eq:boundmukh}, we get 
\begin{equation}\label{eq:bound_mu}
\widehat{\mu}_k\geq\mu_{\min}\quad\text{and}\quad\left\|\nabla h(x_k)\nabla h(x_k)^T\right\|+\widehat{\mu}_k\leq L_0^2+\tau,\quad\text{for }x_k\in\mathbb{B}(x^*,\overline{s}).
\end{equation}
By making $\overline{s}$ smaller if needed, we can guarantee
\begin{equation}\label{eq:choose_s}
\mu_{\min} \geq\frac{2+\sqrt{5}}{4}L\|h(x)\|,\quad\forall x\in\mathbb{B}(x^*,\overline{s}).
\end{equation}
Lipschitz continuity of $h$ and $\overline{s}\leq \mathtt{r}<1$, for all $x\in\mathbb{B}(x^*,\overline{s})$, lead to
\begin{equation}\label{eq:bound_psi}
\psi(x)=\frac{1}{2}\|h(x)-h(x^*)\|^2\leq\frac{L^2}{2}\|x-x^*\|^2\leq\frac{L^2}{2}\|x-x^*\|.
\end{equation}
We now define
$$\Delta:=\frac{2^\theta\kappa L^{2(1-\theta)}(L_0^2+\tau) }
{(1-\theta)\mu_{\min}}, \quad {s}:= \left( \frac{\overline{s}}{1+\Delta}\right)^{\frac{1}{1-\theta}}.$$
From $\overline{s}<1$ and $\theta\in [0,1[$, we obtain $s\leq \overline{s}$.

For $k'\geq \overline{k}_0$, let us choose any $x_{k'}\in\mathbb{B}(x^{*},s)$. Lemma~\ref{Svaiter_lemma} and $d_k=-H_k^{-1}\nabla h(x_{k})h(x_{k})$ imply, for all $k\in \mathbb{N}$,
\begin{equation}\label{eq:conseq_Svaiter}
\psi(x_{k+1})\leq\psi(x_{k})-\frac{1}{2}d_{k}^{T}H_k d_k+\frac{\|d_{k}\|^{2}}{2\widehat{\mu}_{k}}\left(\frac{L^{2}}{16}\|h(x_{k})\|^{2}+L\widehat{\mu}_{k}\|h(x_{k})\|-\widehat{\mu}_{k}^{2}\right).
\end{equation}

Next, let us show by induction that, for $i \in \mathbb{N}$,
\begin{equation}\label{eq:induction}
x_{k'+i}\in\mathbb{B}(x^*,\overline{s}),\quad\|d_{k'+i-1} \| \leq \frac{2\kappa(L_0^2+\tau)}{(1-\theta)\mu_{\min}} \left(\psi(x_{k'+i-1})^{1-\theta}-\psi(x_{k'+i})^{1-\theta}\right)
\end{equation}
It follows from $x_{k'}\in\mathbb{B}(x^*,\overline{s})$ and \eqref{eq:choose_s} that
$$
\widehat{\mu}_{k'} \geq \mu_{\min} \geq\frac{2+\sqrt{5}}{4}L\|h(x_{k'})\|,$$
leading to
\[
\frac{L^{2}}{16}\|h(x_{k'})\|^{2}+L\widehat{\mu}_{k'}\|h(x_{k'})\|-\widehat{\mu}_{k'}^{2}\leq 0.
\]
Then, from~\eqref{eq:conseq_Svaiter}, one can deduce
\begin{equation}\label{eq:psi_x1}
\psi(x_{k'+1})\leq\psi(x_{k'})-\frac{1}{2}d_{k'}^{T}H_{k'} d_{k'}\leq \psi(x_{k'})-\frac{\mu_{\min}}{2}\|d_{k'}\|^2.
\end{equation}
From the convexity of the function $\varphi(t):=-t^{1-\theta}$ with $t>0$, we come to
\begin{equation}\label{eq:convexity_varphi}
\psi(x)^{1-\theta}-\psi(y)^{1-\theta}
\geq  (1-\theta)\psi(x)^{-\theta}\left( \psi(x)-\psi(y)\right),\quad\forall x,y\in\mathbb{R}^m\setminus\Omega.
\end{equation}
This and~\eqref{eq:psi_x1} suggest
\begin{equation}\label{psi_estimate_1}
\psi(x_{k'})^{1-\theta}-\psi(x_{k'+1})^{1-\theta}\geq \frac{(1-\theta)\mu_{\min}}{2} \psi(x_{k'})^{-\theta}\|d_{k'}\|^2.
\end{equation}
It follows from $x_0\in\mathbb{B}(x^*,s)\subseteq\mathbb{B}(x^*,\overline{s})$ and~\eqref{eq:bound_mu} that $\|H_0\|\leq(L_0^2+\tau)$. Hence, by the \L{}ojasiewicz gradient inequality~\eqref{eq:Lojasiewicz_Gradient_Inequality}, we get
\[
\psi(x_{k'})^{\theta} \leq \kappa \|\nabla \psi (x_{k'}) \|
 \leq \kappa \|H_{k'}\| \|d_{k'} \| \leq \kappa (L_0^2+\tau)\|d_{k'} \|.
\]
This,~\eqref{psi_estimate_1}, and~\eqref{eq:bound_psi} yield
\begin{align*}
\|d_{k'} \| &\leq \frac{2\kappa(L_0^2+\tau)}{(1-\theta)\mu_{\min}} \left(\psi(x_{k'})^{1-\theta}-\psi(x_{1})^{1-\theta}\right)\\
&\leq\frac{2\kappa(L_0^2+\tau)}{(1-\theta)\mu_{\min}} \psi(x_{k'})^{1-\theta}\leq \Delta\|x_{k'}-x^*\|^{1-\theta},
\end{align*}
which, proves the second assertion in~\eqref{eq:induction} for $i=1$.
Then, we have
\begin{align*}
\|x_{k'+1}-x^*\| &\leq \|x_{k'}-x^*\|+\|d_{k'}\|
\leq \|x_{k'}-x^*\|+\Delta \|x_{k'}-x^*\|^{1-\theta}\\
&\leq (1+\Delta) \|x_{k'}-x^*\|^{1-\theta}
\leq (1+\Delta) s^{1-\theta} =\overline{s},
\end{align*}
implying $x_{k'+1}\in\mathbb{B}(x^*,\overline{s})$. Now, let us assume that~\eqref{eq:induction} holds for all $i=1,\ldots, k$.
From $x_k\in\mathbb{B}(x^*,\overline{s})$ and~\eqref{eq:choose_s}, it can be deduced
\[
\widehat{\mu}_{k'+k} \geq \mu_{\min} \geq\frac{2+\sqrt{5}}{4}L\|h(x_{k'+k})\|,
\]
leading to
\[
\frac{L^{2}}{16}\|h(x_{k'+k})\|^{2}+L\widehat{\mu}_{k'+k}\|h(x_{k'+k})\|-\widehat{\mu}_{k'+k}^{2}\leq0.
\]
It follows from this and~\eqref{eq:conseq_Svaiter} that
\begin{equation}\label{eq:psi_nonincreasing}
\psi(x_{k'+k+1})\leq \psi(x_{k'+k})-\frac{1}{2} d_{k'+k}^T H_{k'+k} d_{k'+k}\leq\psi(x_{k'+k})-\frac{\mu_{\min}}{2}\|d_{k'+k}\|^2.
\end{equation}
A combination of this inequality and~\eqref{eq:convexity_varphi} leads to
\begin{equation}\label{psi_estimate_k}
\psi(x_{k'+k})^{1-\theta}-\psi(x_{k'+k+1})^{1-\theta}\geq \frac{(1-\theta)\mu_{\min}}{2} \psi(x_{k'+k})^{-\theta}\|d_{k'+k}\|^2
\end{equation}
Further, from $x_{k'+k}\in\mathbb{B}(x^*,\overline{s})$, \eqref{eq:Lojasiewicz_Gradient_Inequality}, and~\eqref{eq:bound_mu}, we obtain
\[
\psi(x_{k'+k})^{\theta} \leq \kappa \|\nabla \psi (x_{k'+k}) \|
 \leq \kappa \|H_{k'+k}\| \|d_{k'+k} \| \leq \kappa (L_0^2+\tau)\|d_{k'+k}\|.
\]
By the latter inequality and~\eqref{psi_estimate_k}, we come to
\[
\|d_{k'+k} \| \leq \frac{2\kappa(L_0^2+\tau)}{(1-\theta)\mu_{\min}} \left(\psi(x_{k'+k})^{1-\theta}-\psi(x_{k'+k+1})^{1-\theta}\right),
\]
proving the second assertion in~\eqref{eq:induction} for $i=k+1$. Then, it follows from~\eqref{eq:bound_psi} that
\begin{align*}
\|x_{k'+k+1}-x^*\| &\leq \|x_{k'}-x^*\|+\sum_{i=k'}^{k'+k}\|d_i\|\\
&\leq \|x_{k'}-x^*\|+\frac{2\kappa(L_0^2+\tau)}{(1-\theta)\mu_{\min}} \sum_{i=k'}^{k'+k} \left(\psi(x_i)^{1-\theta}-\psi(x_{i+1})^{1-\theta}\right)\\
&= \|x_{k'}-x^*\|+\frac{2\kappa(L_0^2+\tau)}{(1-\theta)\mu_{\min}} \left(\psi(x_{k'})^{1-\theta}-\psi(x_{k'+k+1})^{1-\theta}\right)\\
&\leq  \|x_{k'}-x^*\|+\frac{2\kappa(L_0^2+\tau)}{(1-\theta)\mu_{\min}} \psi(x_{k'})^{1-\theta}\\
&\leq (1+\Delta) \|x_{k'}-x^*\|^{1-\theta}
\leq (1+\Delta) s^{1-\theta} =\overline{s}.
\end{align*}
Hence, the first assertion in~\eqref{eq:induction} is valid for $i=k+1$. 

Finally, we are in a position to show that Assertions (ii) is true.
As shown in~\eqref{eq:induction}, $x_k\in\mathbb{B}(x^*,\overline{s})$ for all $k\geq k'$. This and~\eqref{eq:bound_mu}, implies that
$\|H_k\|\leq (L_0^2+\tau)$ for all $k\geq k'$. Hence, for $k\geq k'$, we have
\[
d_k^T H_k d_k=\nabla \psi(x_k)^{T}H_k^{-1}\nabla \psi(x_k)
 \geq \frac{1}{\|H_k\|} \|\nabla \psi(x_k)\|^2 \geq \frac{1}{(L_0^2+\tau)}\|\nabla \psi(x_k)\|^2.
 \]
Then, by~\eqref{eq:psi_nonincreasing}, we get
\[
\psi(x_{k+1})\leq \psi(x_k)-\frac{1}{2(L_0^2+\tau)}\|\nabla \psi(x_k)\|^2.
\]
From this and~\eqref{eq:Lojasiewicz_Gradient_Inequality}, it can be deduced
\[
\psi(x_{k+1})\leq \psi(x_k)-\frac{1}{2\kappa^2(L_0^2+\tau)}\psi(x_k)^{2\theta},\quad \forall~ k\geq k',
\]
which implies that $\{\psi(x_k)\}$ converges to $0$. This ans the H\"older metric subregularity validate the statement of the assertion (ii). 

Applying Lemma~\ref{lem:rate_convergence} with $s_k:=\psi(x_k)$, $\nu:=2\kappa^2(L_0^2+\tau)$ and $\zeta:=2\theta$, we have that the convergence rate are dependent to $\theta$ as claimed in Assertions (iii)-(v). Therefore, the H\"older metric subregularity of $h$ implies that $\{\mathrm{dist}\left(x_{k},\Omega\right)\}$ converges to $0$ with the rate given in (iii)-(v).
\end{proof}

%%%%%%%%%%%%%%%%%%%%%%%%%%%%%%%%%%%%%%%%%%%%%%%%%%%%%%%%
%%%%%%%%%%%%%%%%%%%%%%%%%%%%%%%%%%%%%%%%%%%%%%%%%%%%%%%%
\section{Convergence to a solution of nonlinear systems}
Let us emphasis that the algorithms LMLS and LMTR only guarantee the
convergence of the sequence $\{x_{k}\}$ to a stationary point $x^{*}$ of the merit function
$\psi$, which can be a local non-global minimiser of~\eqref{eq:meritfunc},
i.e.,
\[
\nabla h(x^{*})h(x^{*})=0,~~~h(x^{*})\neq0.
\]
Therefore, the remainder of this section concerns with considering more restrictions
on the mapping $h$ such that the global convergence of $\{x_{k}\}$
to a solution of~\eqref{eq:nonequa} is guaranteed. 

The next theorem extracts some classical results for cases that $\nabla h(x^{*})$ is nonsingular, 
which implies that $x^{*}$ is a solution of~\eqref{eq:nonequa}. Moreover, the worst-case global and 
evaluation complexities to attain solution of~\eqref{eq:nonequa} are provided under the nonsingularity
of $\nabla h(x)\nabla h(x)^{T}$ for all $x\in\mathcal{L}(x_{0})$.
Under the assumption that all accumulation points of $\left\{ x_{k}\right\}$
are solutions of~\eqref{eq:nonequa} and $\nabla h(x^{*})$ is
nonsingular for the accumulation point $x^{*}$, it is proved that the whole sequence
$\left\{ x_{k}\right\} $ converges to the isolated solution $x^{*}$ of~\eqref{eq:nonequa}.

\begin{thm}
\label{thm:posi} Let $\left\{ x_{k}\right\} $ be the sequence generated
by LMLS or LMTR and (A1)-(A3) hold. Then
\begin{description}
\item [{(i)}] if $\nabla h(x^{*})$ is nonsingular at any accumulation point $x^{*}$ of $\left\{ x_{k}\right\} $, 
then $x^{*}$ is a solution
of the nonlinear system~\eqref{eq:nonequa}.
\item [{(ii)}] if the matrix $\nabla h(x)\nabla h(x)^{T}$ is nonsingular for all
$x\in\mathcal{L}(x_{0})$, i.e., there exists $\lambda>0$ such that
$\lambda_{\min}(\nabla h(x)\nabla h(x)^{T})>\lambda$, then, for 
LMLS,  
\begin{equation}
N_{i}(\varepsilon)\leq\lceil\lambda^{-2}c_{2}^{-1}\psi(x_{0})\varepsilon^{-2}+1\rceil\label{eq:nieps2}
\end{equation}
and
\begin{equation}\label{eq:nfeps2}
N_{f}(\varepsilon)\leq\frac{\lceil\lambda^{-2}c_{2}^{-1}\psi(x_{0})\varepsilon^{-2}+1\rceil
(\log(\widehat{\alpha})-\log(\overline{\alpha}))}{\log(\rho)},
\end{equation}
and, for LMTR, 
\begin{equation}
N_{i}(\varepsilon)\leq\lceil\lambda^{-2}c_{3}^{-1}\psi(x_{0})\varepsilon^{-2}+1\rceil\label{eq:nieps3}
\end{equation}
and
\begin{equation} \label{eq:nfeps3}
N_{f}(\varepsilon)\leq \lceil \lambda^{-2}\overline{c}_3^{-1}\psi(x_{0})\varepsilon^{-2}+1\rceil 
\left(\frac{\log(\tau)-\log(\lambda_{min}\mu_{min}))}{\log(\rho_1)}\right).
\end{equation}
\item [{(iii)}] if all accumulation points of $\left\{ x_{k}\right\} $
are solutions of the nonlinear system~\eqref{eq:nonequa}, $x^{*}$is
an accumulation point of $\left\{ x_{k}\right\} $ such that $\nabla h(x^{*})$
is nonsingular, and
\begin{equation}
\lim_{k\rightarrow\infty}\|x_{k+1}-x_{k}\|=0,\label{eq:limxk1}
\end{equation}
then $\left\{ x_{k}\right\} $ converges to $x^{*}$.
\end{description}
\end{thm}

\begin{proof}
For any accumulation point $x^{*}$of $\left\{ x_{k}\right\}$, it follows from Theorem~\ref{thm:global_convergence} 
that $\nabla h(x^{*})h(x^{*})=0$. This, along with the nonsingularity of $\nabla h(x^{*})$, implies Assertion (i).

To prove Assertion (ii), we note that
\[
\|\nabla\psi(x_{k})\|^{2}=h(x_{k})\nabla h(x_{k})\nabla h(x_{k})^{T}h(x_{k})\geq\lambda\|h(x_{k})\|^{2},
\]
i.e., $\|h(x_{k})\|\leq\lambda^{-1/2}\|\nabla\psi(x_{k})\|$. This
and Proposition \ref{thm:complexity} (i)-(ii) give~\eqref{eq:nieps2}
and~\eqref{eq:nfeps2}, respectively. Similarly,  \eqref{eq:nieps3}
and~\eqref{eq:nfeps3} follow from this inequality and
Proposition \ref{thm:complexity1} (i)-(ii).

In order to prove Assertion (iii), let us assume that all accumulation points of
$\left\{ x_{k}\right\} $ are solutions of~\eqref{eq:nonequa}, $x^{*}$is
an accumulation point such that $\nabla h(x^{*})$ is nonsingular,
and~\eqref{eq:limxk1} holds. From the inverse function theorem and
the nonsingularity of $\nabla h(x^{*})$, there exists a neighborhood
around 0 such that $h$ is invertible. Therefore, there exists a neighborhood
$\mathbb{B}(x^*,\mathtt{r}_1)$ for $\mathtt{r}_1>0$ such that
\[
h(x)\neq0,\quad\forall\,x\in\mathbb{B}(x^{*},r_{1})\:\mbox{and}\:x\neq x^{*},
\]
implying

\[
\|h(x)\|>0,\quad\forall\,x\in\mathbb{B}(x^{*},r_{1})\:\mbox{and}\:x\neq x^{*}.
\]
Since $x^{*}$is an accumulation point of $\left\{ x_{k}\right\} $,
$\mathbb{B}(x^{*},r_{1})$ contains an infinite number of iteration
points of $\left\{ x_{k}\right\} $. It remains to show that there
exists $k_{2}\in\mathbb{N}$ such that $x_{k}\in\mathbb{B}(x^{*},r_{1})$, for all $k\geq k_{2}$. 
Hence, for an arbitrary $\varepsilon\in(0,r_{1})$,
the set $\mathbb{B}(x^{*},r_{1})-\mathbb{B}(x^{*},\varepsilon)$ involves
only a finite number of iterations of $\left\{ x_{k}\right\} $, i.e.,
there exists $k_{3}\in\mathbb{N}$ such that
\[
x_{k}\in\mathbb{B}(x^{*},\varepsilon),\quad\forall\,k\geq k_{3}.
\]
It follows from~\eqref{eq:limxk1} that there exists $k_4\in\mathbb{N}$
such that
\[
\|x_{k+1}-x_{k}\|\leq\delta-\varepsilon,\quad\forall\,k\geq k_{3}.
\]
Let us set $k_{2}:=\max\left\{ k_{3},k_{4}\right\} $ leading
to
\[
\|x_{k+1}-x^{*}\|\leq\|x_{k+1}-x_{k}\|+\|x_{k}-x^{*}\|\leq\delta,\quad\forall\,k\geq k_{2},
\]
giving the result.
\end{proof}

The mapping $h$ is called strictly monotone on $\mathbb{R}^{m}$
if
\[
(h(x)-h(y))^{T}(x-y)>0,\quad\forall\,x,y\in\mathbb{R}^{m}.
\]
In addition, the mapping $h$ is called strictly duplomonotone with
constant $\overline{\tau}>0$ if
\[
(h(x)-h(x-\tau h(x)))^{T}h(x)>0,\quad\forall\,x\in\mathbb{R}^{m},\tau\in(0,\overline{\tau}],
\]
whenever $h(x)\neq0$; see \cite{artacho_globally_2014,ortega_iterative_2000}. In the next result, we will show that if the
mapping $h$ or $-h$ is strictly monotone (duplomonotone), then the
sequence $\left\{ x_{k}\right\} $ generated by LMLS
converges to the unique solution of the nonlinear system~\eqref{eq:nonequa}.

\begin{thm}
\label{thm:mono}Let $\left\{ x_{k}\right\} $ be the sequence generated
by LMLS or LMTR and (A1)-(A3) hold.
\begin{description}
\item [{(i)}] If the mapping $h$ or $-h$ is strictly monotone, then $\left\{ x_{k}\right\} $
converges to the unique solution of the nonlinear system~\eqref{eq:nonequa}.
\item [{(ii)}] If the mapping $h$ or $-h$ is strictly duplomonotone,
then $\left\{ x_{k}\right\} $ converges to a solution of the nonlinear
system~\eqref{eq:nonequa}.
\end{description}
\end{thm}

\begin{proof}
In order to prove Assertion (i), let $h$ or $-h$ be strictly monotone and $x^{*}$ be an accumulation point
of $\left\{ x_{k}\right\} $. If $h$ is strictly monotone, for the points $x$ and $x+th$ with $t>0$ and $h\in\mathbb{R}^{m}$,
we can deduce
\[
\begin{split}
z^{T}\nabla h(x)z=z^{T}\left(\lim_{t\rightarrow0}\frac{h(x+tz)-h(x)}{t}\right)
=\lim_{t\rightarrow0}\left(z^{T}\frac{h(x+tz)-h(x)}{t}\right)>0,\quad\forall\,x,z\in\mathbb{R}^{m}.
\end{split}
\]
If $-h$ is strictly monotone, then
\[
\begin{split}
z^{T}\nabla h(x)z  =-z^{T}\left(\lim_{t\rightarrow 0}\frac{h(x+tz)-h(x)}{t}\right)
=-\lim_{t\rightarrow0}\left(z^{T}\frac{h(x+tz)-h(x)}{t}\right)<0,\quad\forall\,x,z\in\mathbb{R}^{m}.
\end{split}
\]
By setting $z=h(x^{*})$ and $x=x^{*}$ in the last two inequalities, we get
\[
h(x^{*})^{T}\nabla h(x^{*})h(x^{*})\neq0,\quad\forall\,h(x^{*})\in\mathbb{R}^{m},\:h(x^{*})\neq0.
\]
This, \eqref{eq:limngrad}, and~\eqref{eq:limngrad1} imply $h(x^{*})=0$.

To prove Assertion (ii), let $h$ be strictly duplomonotone, which leads to

\begin{eqnarray*}
h(x)^{T}\nabla h(x)h(x) & = & \left(\lim_{\tau\rightarrow0}\frac{h(x-\tau h(x))-h(x)}{-\tau}\right)^{T}h(x)\\
 & = & \lim_{\tau\rightarrow0}\left(\frac{h(x)-h(x-\tau h(x))}{\tau}\right)^{T}h(x)>0,\quad\forall\,x\in\mathbb{R}^{m},\tau\in(0,\overline{\tau}].
\end{eqnarray*}
If $-h$ is strictly duplomonotone, then

\begin{eqnarray*}
-h(x)^{T}\nabla h(x)h(x) & = & -\left(\lim_{\tau\rightarrow0}\frac{h(x+\tau h(x))-h(x)}{\tau}\right)^{T}h(x)\\
 & = & \lim_{\tau\rightarrow0}\left(\frac{-h(x)+h(x-\tau(-h(x)))}{\tau}\right)^{T}(-h(x))>0,\quad\forall\,x\in\mathbb{R}^{m},\tau\in(0,\overline{\tau}].
\end{eqnarray*}
The result follows from the last two inequalities at $x=x^{*}$, \eqref{eq:limngrad}, and~\eqref{eq:limngrad1}.
\end{proof}
Note that the strict monotonicity of $h$ does not implies the positive
definiteness of $\nabla h(x)$. Therefore the results of Theorem \ref{thm:mono} (i)
is not a trivial consequence of Theorem \ref{thm:posi} (i).

%%%%%%%%%%%%%%%%%%%%%%%%%%%%%%%%%%%%%%%%%%%%%%%%%%%%%%%%%%%%%%%%%%%%%%%%
%%%%%%%%%%%%%%%%%%%%%%%%%%%%%%%%%%%%%%%%%%%%%%%%%%%%%%%%%%%%%%%%%%%%%%%%
\section{Application to biochemical reaction networks}\label{sec.numapp}
In this section, we use the following notation: 
$\mathbb{Z}_+^{m\times n}:= \left\{ A\in \mathbb{Z}^{m\times n} ~|~ A_{ij}\geq 0,~i=1,\ldots,m,~j=1,\ldots,n\right\}$,
$\mathbb{R}_+^{m}:= \left\{ a\in \mathbb{R}^{m} ~|~ a_i\geq 0,~i=1,\ldots,m\right\}$, and
$\mathbb{R}_{++}^{m}:= \left\{ a\in \mathbb{R}^{m} ~|~ a_i> 0,~i=1,\ldots,m\right\}$.
Let us consider a biochemical reaction network with $m$ molecular
species and $n$ reversible elementary reactions\footnote{An elementary reaction is a 
chemical reaction for which no intermediate
molecular species need to be postulated in order to describe the chemical
reaction on a molecular scale.}. We define {\it forward} and \emph{ reverse stoichiometric matrices}, $F,R\in\mathbb{\mathbb{Z}}_{+}^{m\times n}$,
respectively, where $F_{ij}$ denotes the \emph{stoichiometry}\footnote{Reaction stoichiometry is a quantitative relationship between the
relative quantities of molecular species involved in a single chemical
reaction.} of the $i^{th}$ molecular species in the $j^{th}$ forward reaction
and $R_{ij}$ denotes the stoichiometry of the $i^{th}$ molecular
species in the $j^{th}$ reverse reaction.  We assume that \emph{every
reaction conserves mass}, i.e., there exists at least a positive
vector $l\in\mathbb{R}_{++}^{m}$ such that $(R-F)^{T}l=0$; cf. \cite{gevorgyan2008detection}. 
The matrix $N:=R-F$ represents net reaction stoichiometry and may be viewed as 
an incidence matrix of a directed hypergraph; see  \cite{Klamt2009}. In practice,
there are less molecular species than net reactions ($m < n$). We assume the
cardinality of each row of $F$ and $R$ is at least one, and the cardinality
of each column of $R-F$ is at least two. The matrices $F$ and $R$ are sparse and 
the sparsity pattern depends on the particular biochemical reaction network being modeled. 
It is here assumed that $\text{rank}([F,R])=m$, which is a requirement for kinetic consistency; 
cf.~\cite{Fleming20161}. 

Let $c\in\mathbb{R}_{++}^{m}$ be a vector of molecular
species concentrations. For nonnegative elementary kinetic
parameters $k_{f},k_{r}\in\mathbb{R}_{+}^{n}$, \emph{elementary
reaction kinetics} for forward and reverse elementary reaction rates
as $s(k_{f},c):=\exp(\ln(k_{f})+F^{T}\ln(c))$ and $r(k_{r},c):=\exp(\ln(k_{r})+R^{T}\ln(c))$,
respectively, where $\exp(\cdot)$ and $\ln(\cdot)$ denote the respective
componentwise functions; see, e.g.,~\cite{artacho_accelerating_2015,Fleming20161}.
Then, the system of differential equations
\begin{eqnarray}
\frac{dc}{dt} & \equiv & N(s(k_{f},c)-r(k_{r},c)) \label{eq:dcdt2}\\
 & = & N\left(\exp(\ln(k_{f})+F^{T}\ln(c)\right)-\exp\left(\ln(k_{r})+R^{T}\ln(c))\right)
 =:-f(c).\nonumber
\end{eqnarray}
shows the deterministic dynamical equation for time evolution of molecular
species concentration. A vector $c^{*}$ is called a \emph{steady state} if and only if 
$f(c^{*})=0$. Hence, $c^{*}$ is a steady state of the biochemical system
if and only if
\[
s(k_{f},c^{*})-r(k_{r},c^{*})\in\mathcal{N}(N),
\]
where $\mathcal{N}(N)$ stands for the null space of $N$. The set of steady states 
$\Omega_1=\left\{ c\in\mathbb{R}_{++}^{m},\,f(c)=0\right\}$
will be unchanged if $N$ is replaced by a matrix $\bar{N}$
with the same kernel. Suppose that $\bar{N}\in\mathbb{Z}^{r\times n}$
is the submatrix of $N$ whose rows are linearly independent, then 
$\mathrm{rank}\left(\bar{N}\right)=\mathrm{rank}(N)\eqqcolon r.$
If one replaces $N$ by $\bar{N}$ and transforms~\eqref{eq:dcdt2}
into logarithmic scale, by letting $x\coloneqq\ln(c)\in\mathbb{R}^{m}$,
$k\coloneqq[\ln(k_{f})^{T},\,\ln(k_{r})^{T}]^{T}\in\mathbb{R}^{2n}$,
then the right-hand side of~\eqref{eq:dcdt2} can be translated to
\begin{equation}
\bar{f}(x):=\left[\bar{N},-\bar{N}\right]\exp\left(k+[F,\,R]^{T}x\right),\label{eq:f(x)}
\end{equation}
where $\left[\,\cdot\thinspace,\cdot\,\right]$ stands for the horizontal
concatenation operator.

Let $L\in\mathbb{R}^{m-r,m}$ be a basis for the left nullspace
of $N$, i.e., $L^TN=0$, where $\mathrm{rank}(N)=r$ and $\mathrm{rank}(L)=m-r$.
The system satisfies \emph{moiety conservation} if for
any initial concentration $c_{0}\in\mathbb{R}_{++}^{m}$, it holds
\[
L\,c=L\,\mathrm{exp}(x)=l_{0},
\]
where $l_{0}\in\mathbb{R}_{++}^{m}.$ It is possible to compute $L$ such that each corresponds 
to a structurally identifiable conserved moiety in a biochemical reaction network; 
cf.~\cite{Haraldsdttir2016}. Therefore, finding the \emph{moiety conserved
steady state} of a biochemical reaction network is equivalent to finding a
zero of the mapping 
\begin{equation}\label{eq:steadyStateEquation}
h:\mathbb{R}^m\rightarrow\mathbb{R}^n~~ \mathrm{with}~~ h(x):=\left(\begin{array}{c}
\bar{f}(x)\\
L\,\mbox{exp}(x)-l_{0}
\end{array}\right).
\end{equation}
It was shown by the authors in Section 4.1 of \cite{Ahookhosh_2017} that the merit function $\psi$ satisfies
\L{}ojasiewicz gradient inequality (with an exponent $\theta\in [0,1[$) and
the mapping $h$ is H\"older metrically subregular at $(x^{*},0)$, i.e., the assumption (A1) holds.

%%%%%%%%%%%%%%%%%%%%%%%%%%%%%%%%%%%%%%%%%%%%%%%%%%%%
\subsection{Computational results}\label{sec:numexp}
We find zeros of the mapping \eqref{eq:steadyStateEquation} with a set of real-world biological data
using LMLS and LMTR. In details, we compare the performance of LMLS and LMTR with some state-of-the-art 
algorithms on a set of 21 biochemical reaction networks given in Table \ref{t.biologicalModel}. In Section 4.2 of 
\cite{Ahookhosh_2017}, it is computationally shown that $\nabla h$ is rank-deficient or ill-conditioned
at zeros of the mapping $h$ \eqref{eq:steadyStateEquation} for these biological models. This clearly
justifies the reason of unsuccessful performance of many existing algorithms 
(e.g., gradient descent, Gauss-Newton, and trust-region methods) and vindicates
the development of the two adaptive Levenberg-Mardquart methods (LMLS and LMTR) 
for such difficult problems.

%%%%%%%%%%%%%%%%%
\begin{table}[htbp]
\caption{The list of 21 biological models, where the stoichiometric matrix $N$ is $m\times n$ and $rank$ is the rank of the matrix $N$.}
\label{t.biologicalModel}
\begin{center}\footnotesize
\renewcommand{\arraystretch}{1.2}
\begin{tabular}{|l|rrr|l|rrr|}
\hline
\multirow{2}{*}{Model} \hspace{10mm}& \multirow{2}{*}{$m$} \hspace{10mm}& \multirow{2}{*}{$n$} \hspace{10mm}& 
\multirow{2}{*}{$rank$} \hspace{10mm}& \multirow{2}{*}{Model} \hspace{10mm}& \multirow{2}{*}{$m$} \hspace{10mm}& 
\multirow{2}{*}{$n$} \hspace{5mm}& \multirow{2}{*}{$rank$} \\
&&&&&&&\\
\hline
1. Ecoli\_core &72&73&61 & 12. iMB745 &525&598&490\\
2. iAF692 &462&493&430 & 13. iNJ661 &651&764&604\\
3. iAF1260 &1520&1931&1456 & 14. iRsp1095 &966&1042&921\\
4. iBsu1103 &993&1167&956 & 15. iSB619 &462&508&435\\
5. iCB925 &415&558&386 & 16. iTH366 &583&606&529\\
6. iIT341 &424&428&392 & 17. iTZ479\_v2 &435&476&415\\
7. iJN678 &641&669&589 & 18. iYL1228 &1350&1695&1280\\
8. iJN746 &727&795&700 & 19. L\_lactis\_MG1363 &483&491&429\\
9. iJO1366 &1654&2102&1582 & 20. Sc\_thermophilis\_rBioNet &348&365&320\\
10. iJP815 &524&595&501& 21. T\_Maritima &434&470&414\\
11. iJR904 &597&757&564 & & & &\\
\hline
\end{tabular}
\end{center}
\end{table}
%%%%%%%%%%%%%%%

All codes are written in MATLAB and runs
are performed on a Dell Precision Tower 7000 Series 7810 (Dual Intel
Xeon Processor E5-2620 v4 with 32 GB RAM). We compare LMLS and LMTR with
\renewcommand{\labelitemi}{$\bullet$}
\begin{itemize}
\item LM-YF: a Levenberg--Marquard line search method with $\mu_{k}=\|h(x_{k})\|^{2}$, 
given by Yamashita and Fukushima~\cite{alefeld_rate_2001};
\item LM-FY: a Levenberg--Marquardline search method with $\mu_{k}=\|h(x_{k})\|$, 
given by Fan and Yuan~\cite{fan_quadratic_2005};
\item LevMar: a Levenberg--Marquard trust-region method with $\mu_{k}=\|\nabla h(x_{k})h(x_{k})\|$, 
given by Ipsen et al.~\cite{IpsKP}.
\end{itemize}
The codes of LMLS and LMTR
are publicly available as a part of the COBRA Toolbox v3.0 \cite{heirendt_creation_2018}.
Users can pass the solver name to the parameter structure of the MATLAB function {\tt optimizeVKmodels.m}.
For both LMLS and LMTR, on the basis of our experiments with 
the mapping~\eqref{eq:steadyStateEquation}, we set $\omega_k:=1-\xi_k$ and
\begin{equation}\label{eq:xi}
\xi_{k}:=\left\{ \begin{array}{ll}
0.95 & ~\mathrm{if}~(0.95)^{k}>10^{-2},\\
\max\left((0.95)^{k},10^{-10}\right) & ~\mathrm{otherwise},
\end{array}\right.
\end{equation}
implying $\xi_k \in [10^{-10}, 0.95]$. We here use the starting point $x_{0}=0$ and consider the stopping
criterion
\begin{equation}
\|h(x_{k})\|\leq\max(10^{-6},10^{-12}\|h(x_{0})\|)~~\mathrm{or}~~\|\nabla\psi(x_{k})\|\leq\max(10^{-6},10^{-12}\|\nabla\psi(x_{0})\|),\label{eq:stopcrit}
\end{equation}
cf. \cite{BelCGMT}. We stop the algorithms if either~\eqref{eq:stopcrit} holds or the maximum number of iterations
(say 10,000 for tuning $\eta$ and 100,000 for the comparison) is reached. While LMLS uses the parameters
\[
\overline{\alpha}=1,~ \rho=0.5,~ \sigma=10^{-2},~ 
\theta_{min}=0, ~ \theta_{max}=0.95, ~\theta_k=0.95,
\]
LMTR employs the parameters
\[
\rho_1=2,~ \rho_2=0.5,~\upsilon_1=10^{-4},~\upsilon_1=0.9,~ \lambda_0=10^{-2},~
~\mu_{\min}=10^{-8},\theta_{min}=0, ~ \theta_{max}=0.95,~\theta_k=0.95.
\]

In our comparison, $N_i$, $N_{f}$ and $T$ denote the total number of iterations,
the total number of function evaluations, and the running time, respectively. To 
illustrate the results, we used the Dolan and Mor\'e performance profile~\cite{Dolan2002} 
with the performance measures $N_{f}$ and $T$. In this procedure, the performance
of each algorithm is measured by the ratio of its computational outcome
versus the best numerical outcome of all algorithms. This performance
profile offers a tool to statistically compare the performance of algorithms. Let $\mathcal{S}$ be a set of all algorithms
and $\mathcal{P}$ be a set of test problems. For each problem $p$
and algorithm $s$, $t_{p,s}$ denotes the computational outcome with respect to the performance index, which is used in the definition of the performance ratio
\begin{equation}
r_{p,s}:=\frac{t_{p,s}}{\min\{t_{p,s}:s\in\mathcal{S}\}}.\label{eq:perf}
\end{equation}
If an algorithm $s$ fails to solve a problem $p$, the procedure
sets $r_{p,s}:=r_\text{failed}$, where $r_\text{failed}$ should be strictly
larger than any performance ratio~\eqref{eq:perf}. Let $n_p$ be the number of problems in the
experiment. For any factor
$\tau\in\mathbb{R}$, the overall performance of an algorithm $s$ is given by
\[
\rho_{s}(\tau):=\frac{1}{n_{p}}\textrm{size}\{p\in\mathcal{P}:r_{p,s}\leq\tau\}.
\]
Here, $\rho_{s}(\tau)$ is the probability that a performance ratio
$r_{p,s}$ of an algorithm $s\in\mathcal{S}$ is within a factor $\tau$
of the best possible ratio. The function $\rho_{s}(\tau)$ is a distribution
function for the performance ratio. In particular, $\rho_{s}(1)$
gives the probability that an algorithm $s$ wins over all other considered
algorithms, and $\lim_{\tau\rightarrow r_{\mathrm{failed}}}\rho_{s}(\tau)$
gives the probability that algorithm $s$ solves all considered
problems. Therefore, this performance profile can be considered as
a measure of efficiency among all considered algorithms. In Figures \ref{fig:tuning_eta} 
and \ref{fig:fan_vs_iter}, the number $\tau$ is represented in the $x$-axis, while 
$P(r_{p,s}\leq\tau:1\leq s\leq n_{s})$ is shown in the $y$-axis.

First, let us tune the parameter $\eta$ to get the best performance of 
LMLS and LMTR. To do so, we consider several versions of these algorithms
corresponding to several levels of the parameter $\eta$ ($\eta=0.6,~ 0.8,~1.0,~1.2,~1.4$) and compare the results in 
Figure \ref{fig:tuning_eta}. From this figure,
it is clear that $\eta=1.2$ attains the best results for both LMLS and LMTR. Therefore,
we use $\eta=1.2$ for finding a zero of the mapping $h$ defined in (\ref{eq:steadyStateEquation}); 
however, to solve a different mappings, one may tune this parameter carefully 
before any practical usage.

%%%%%%%%%%%%%%%%%
\begin{figure}[H]
\,\centering%
\begin{minipage}[b][1\totalheight][t]{0.33\columnwidth}%
\subfigure[LMLS, the number of iterations $N_i$]
{\label{fig:AnotherFig-3}\foreignlanguage{english}{\includegraphics[width=7.5cm]{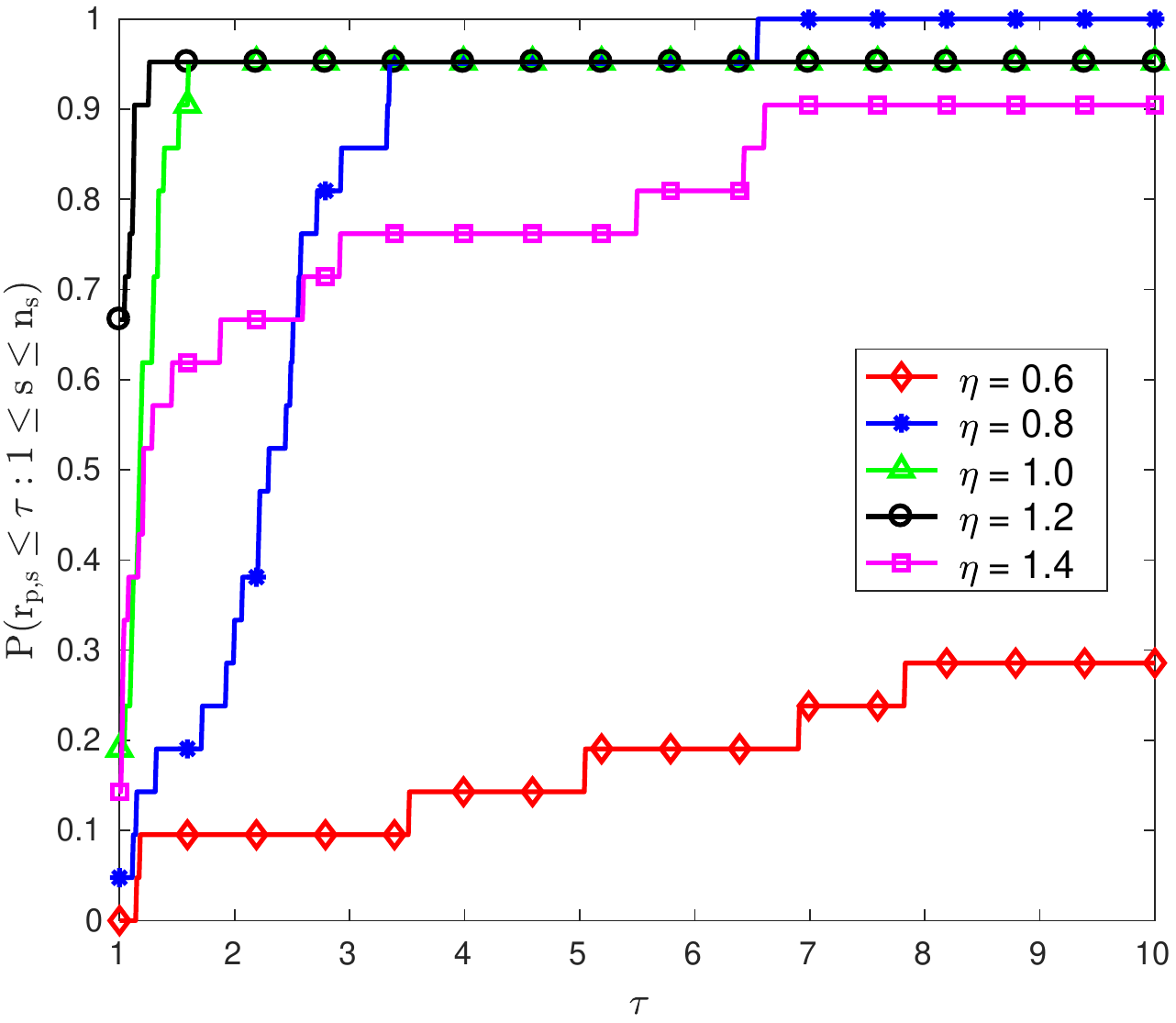}}}%
\end{minipage}\hfill{}%
\begin{minipage}[b][1\totalheight][t]{0.33\columnwidth}%
\subfigure[LMTR, the number of function evaluations $N_i$]
{\foreignlanguage{english}{\includegraphics[width=7.5cm]{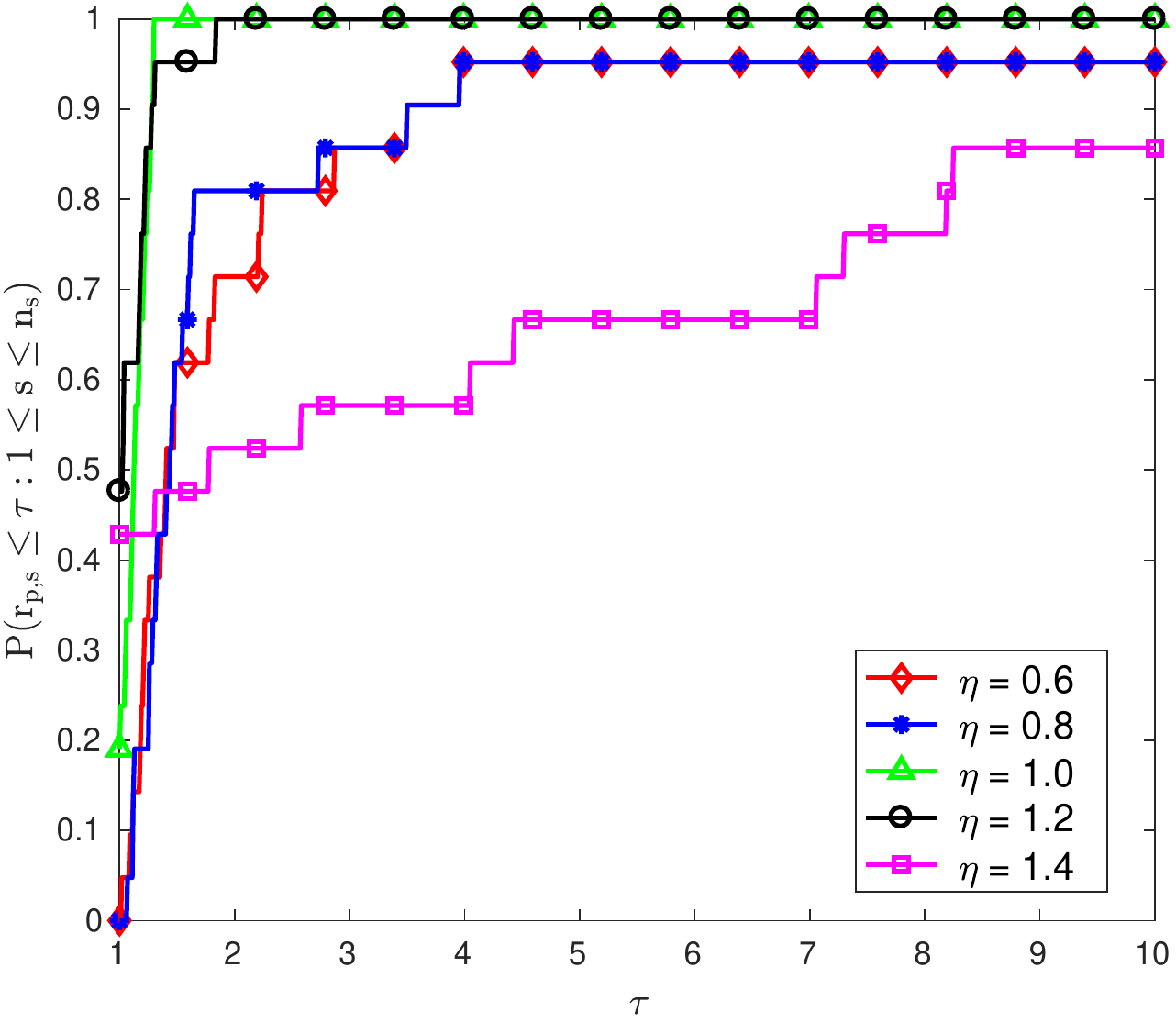}}\label{fig:AnotherFig2-3}}%
\end{minipage}\hspace{2.5cm}

\centering%

\begin{minipage}[b][1\totalheight][t]{0.33\columnwidth}%
\subfigure[LMLS, the number of function evaluations $N_f$]
{\label{fig:AnotherFig-3-1}\foreignlanguage{english}{\includegraphics[width=7.5cm]{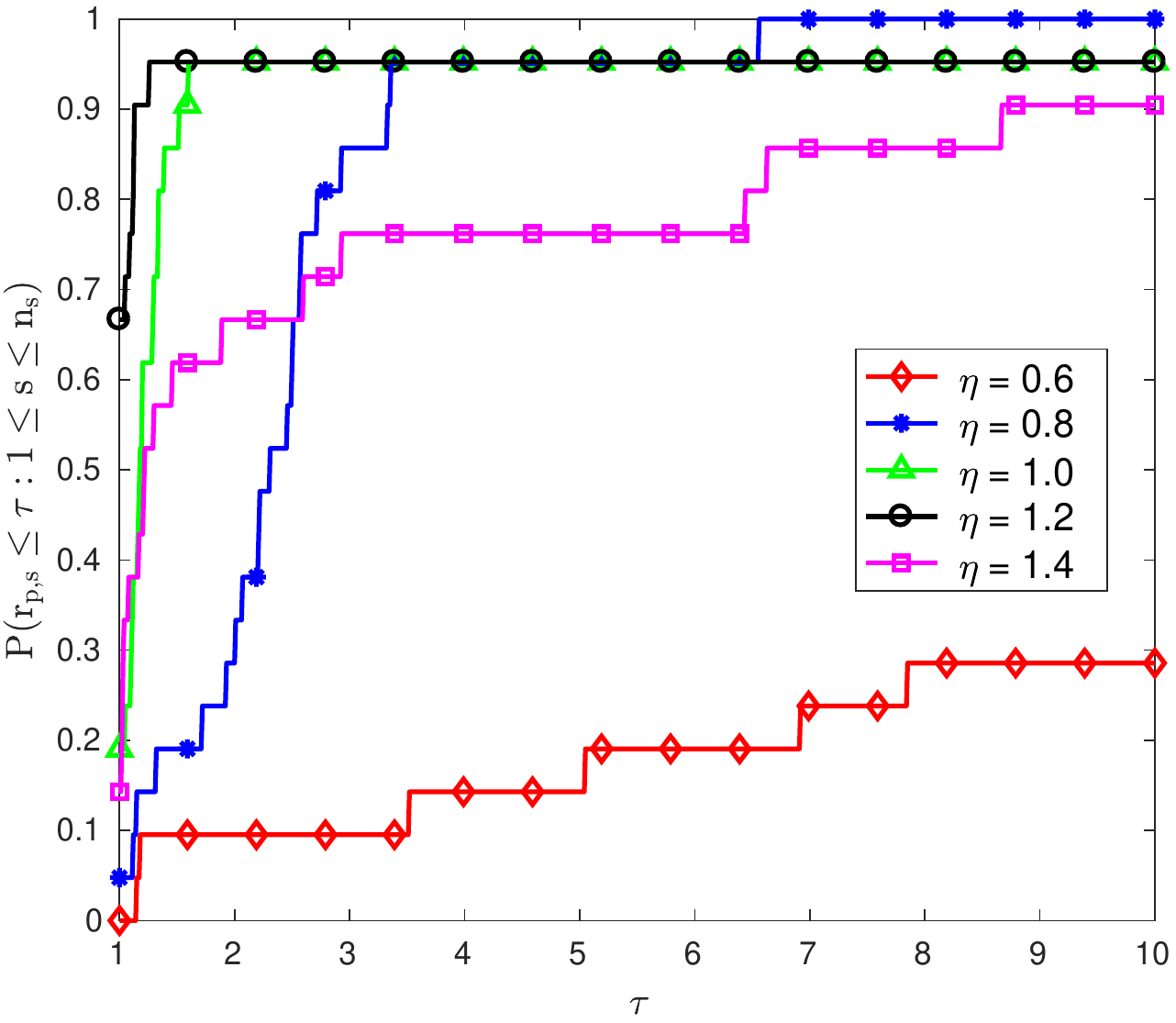}}}%
\end{minipage}\hfill{}%
\begin{minipage}[b][1\totalheight][t]{0.33\columnwidth}%
\subfigure[LMTR, the number of function evaluations $N_f$]{\foreignlanguage{english}
{\includegraphics[width=7.5cm]{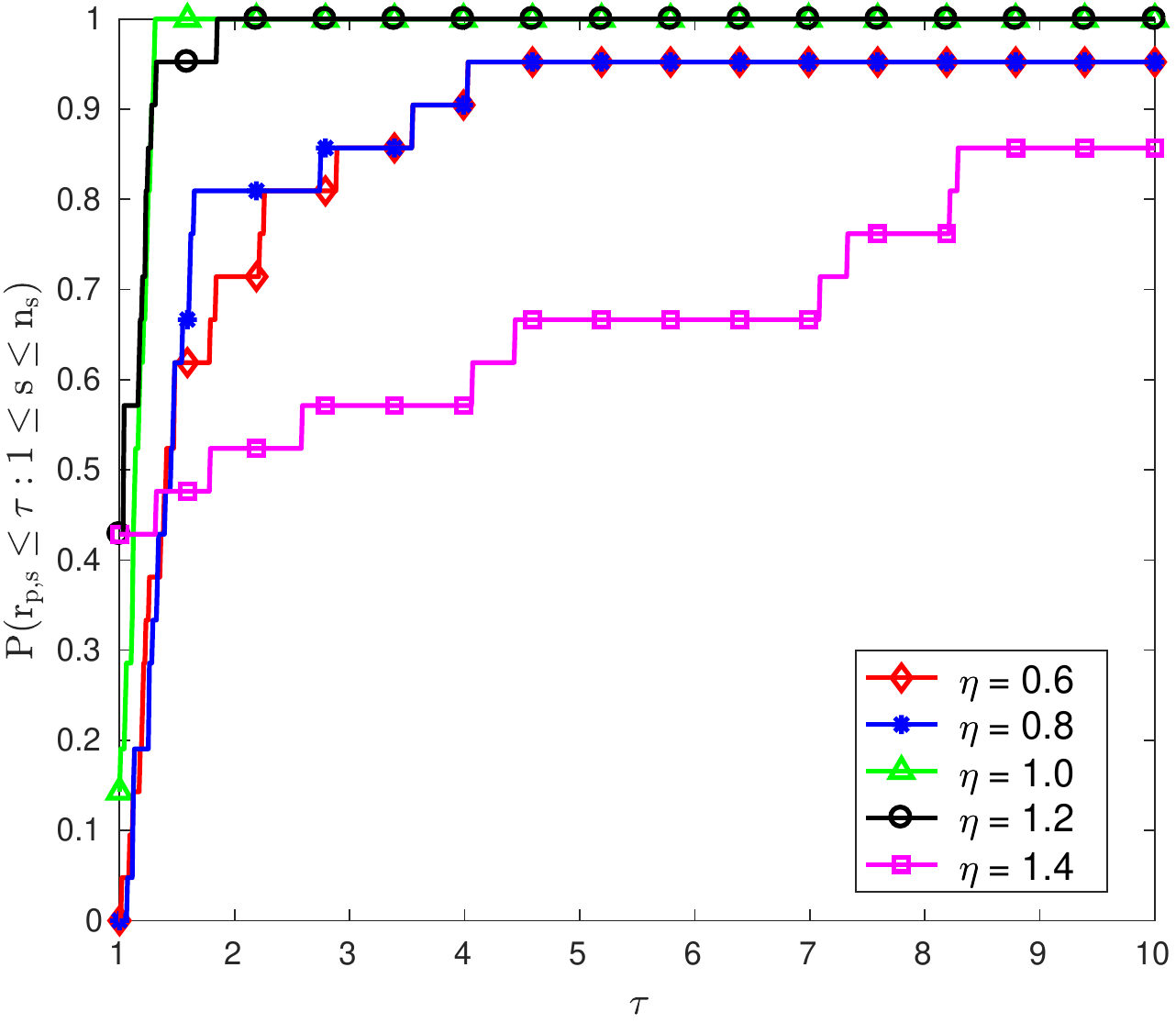}}\label{fig:AnotherFig2-3-1}}%
\end{minipage}\hspace{2.5cm}

\centering%

\begin{minipage}[b][1\totalheight][t]{0.33\columnwidth}%
\subfigure[LMLS, the running time $T$]{\label{fig:AnotherFig-3-1}\foreignlanguage{english}
{\includegraphics[width=7.5cm]{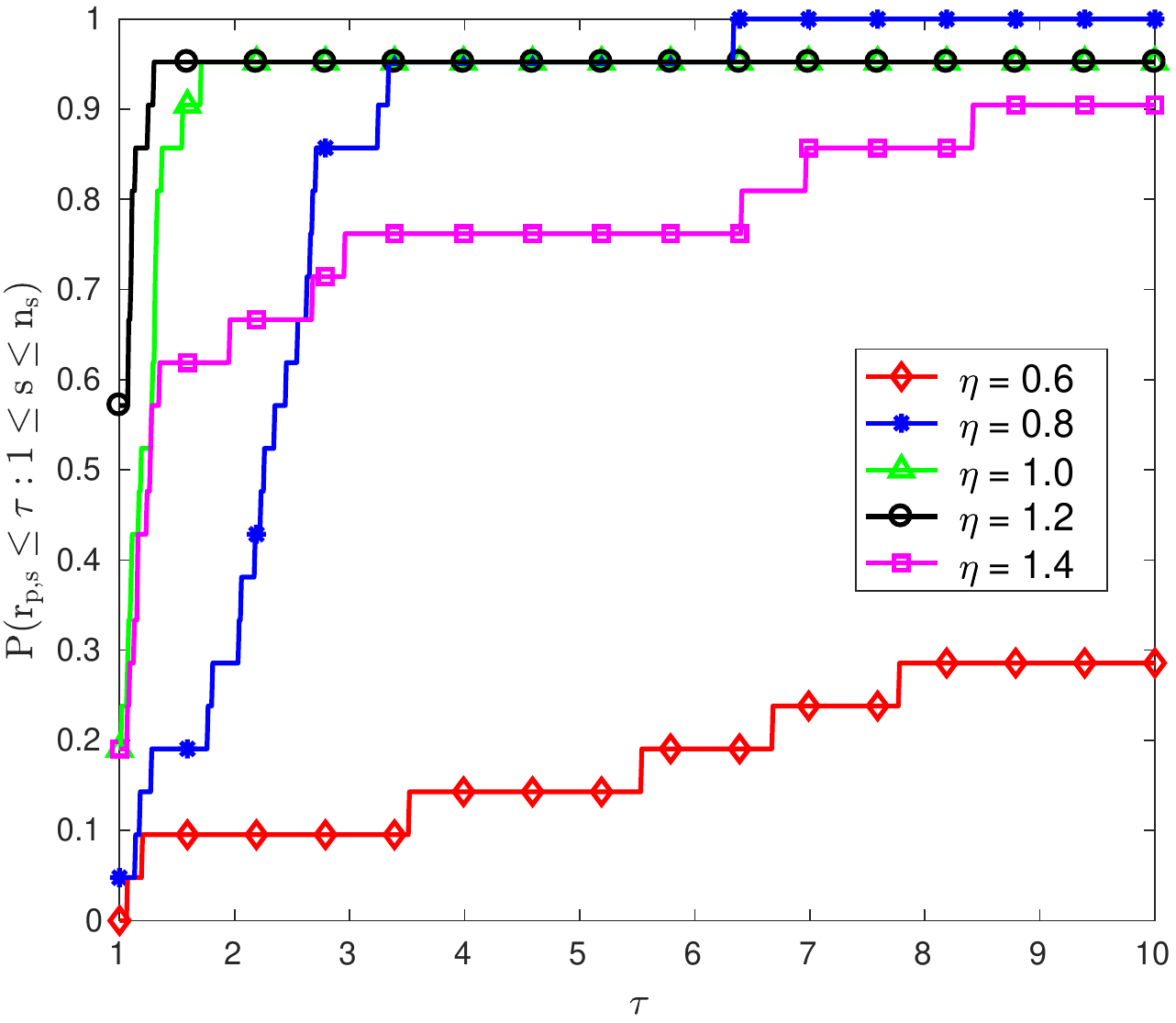}}}%
\end{minipage}\hfill{}%
\begin{minipage}[b][1\totalheight][t]{0.33\columnwidth}%
\subfigure[LMTR, the running time $T$]{\foreignlanguage{english}
{\includegraphics[width=7.5cm]{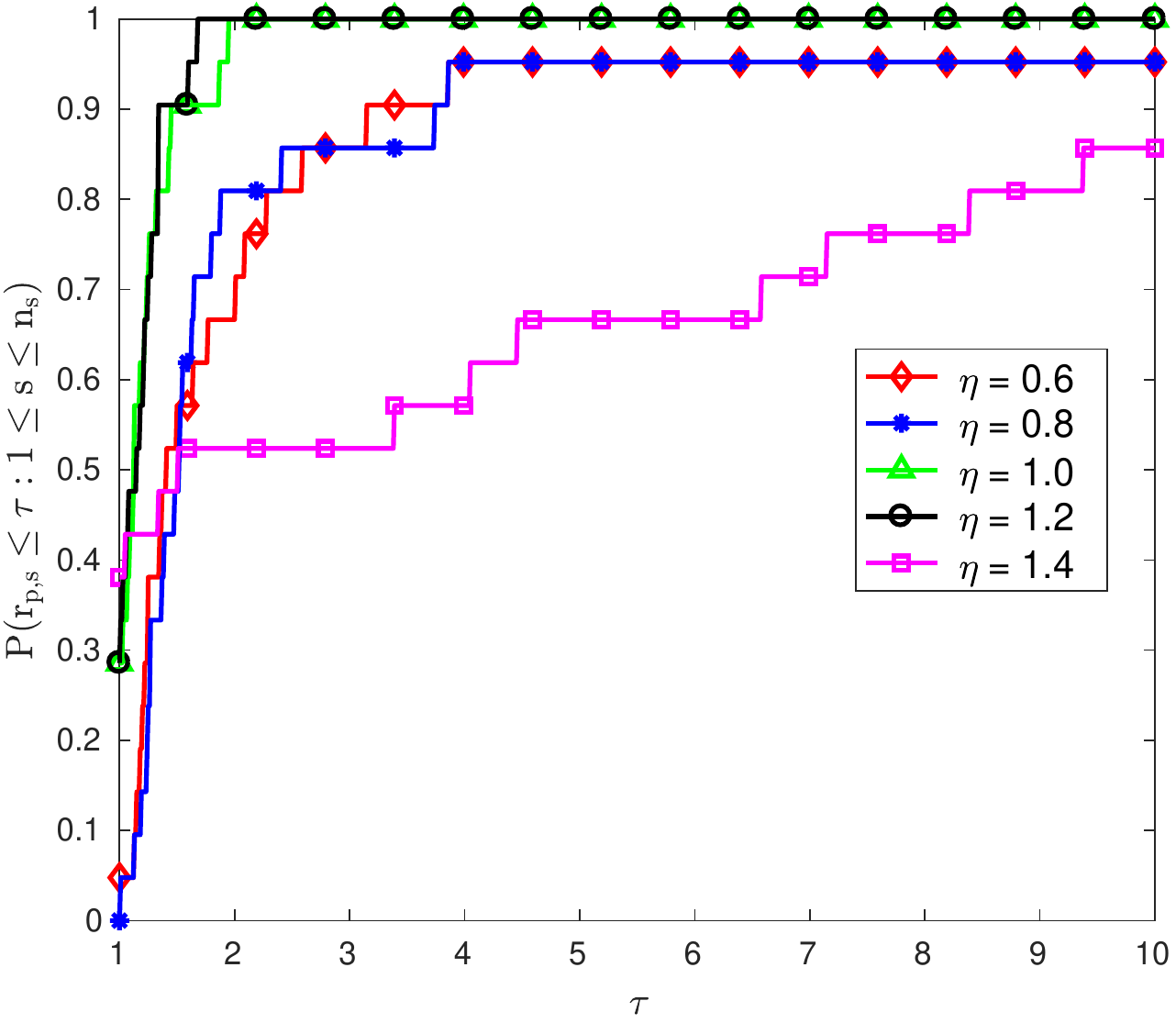}}\label{fig:AnotherFig2-3-1}}%
\end{minipage}\hspace{2.5cm}

\selectlanguage{english}%
\caption{Performance profile for the number of iterations ($N_i$), the number of function evaluations ($N_{f}$),
and the running time ($T$) of LMLS and LMTR to tune the parameter $\eta$, with
$\eta\in\{0.6, 0.8, 1.0, 1.2, 1.4\}$. The best performance is attained by $\eta=1.2$ for both methods.\label{fig:tuning_eta}}
\selectlanguage{american}%
\end{figure}
%%%%%%%%%%%%

Next, we report the results of a comparison among LM-YF, LM-FY, LevMar, LMLS, and LMTR
for finding a zero of $h$ (\ref{eq:steadyStateEquation})
with respect to the total number of iterations ($N_i$),
the total number of function evaluations ($N_f$), the mixed measure $N_f+3N_i$,
and the running time ($T$) in Figure~\ref{fig:per_pro_global}.
From this figure, it can be seen that LMLS and LMTR outperform the 
others substantially with respect to all considered measures. Moreover, LMTR
solves the problems even faster than LMLS; however, the slope of curve of LMLS
indicates that its performance is much better than LM-YF, LM-FY, and LevMar, and
its performance is close to the performance of LMTR. Surprisingly, both LMLS and LMTR 
are convergent to a zero of the mapping $h$ (\ref{eq:steadyStateEquation}) not to
a stationary point of the merit function $\psi$ given by \eqref{eq:meritfunc}.
This clearly show the potential of LMLS and LMTR for finding the moiety conserved
steady state of biochemical reaction networks.

%%%%%%%%%%%%%%%%%
\begin{figure}[H]
\,\centering%
\begin{minipage}[b][1\totalheight][t]{0.33\columnwidth}%
\subfigure[the number of iterations $N_i$]{\label{fig:AnotherFig-3}
\foreignlanguage{english}{\includegraphics[width=7.5cm]{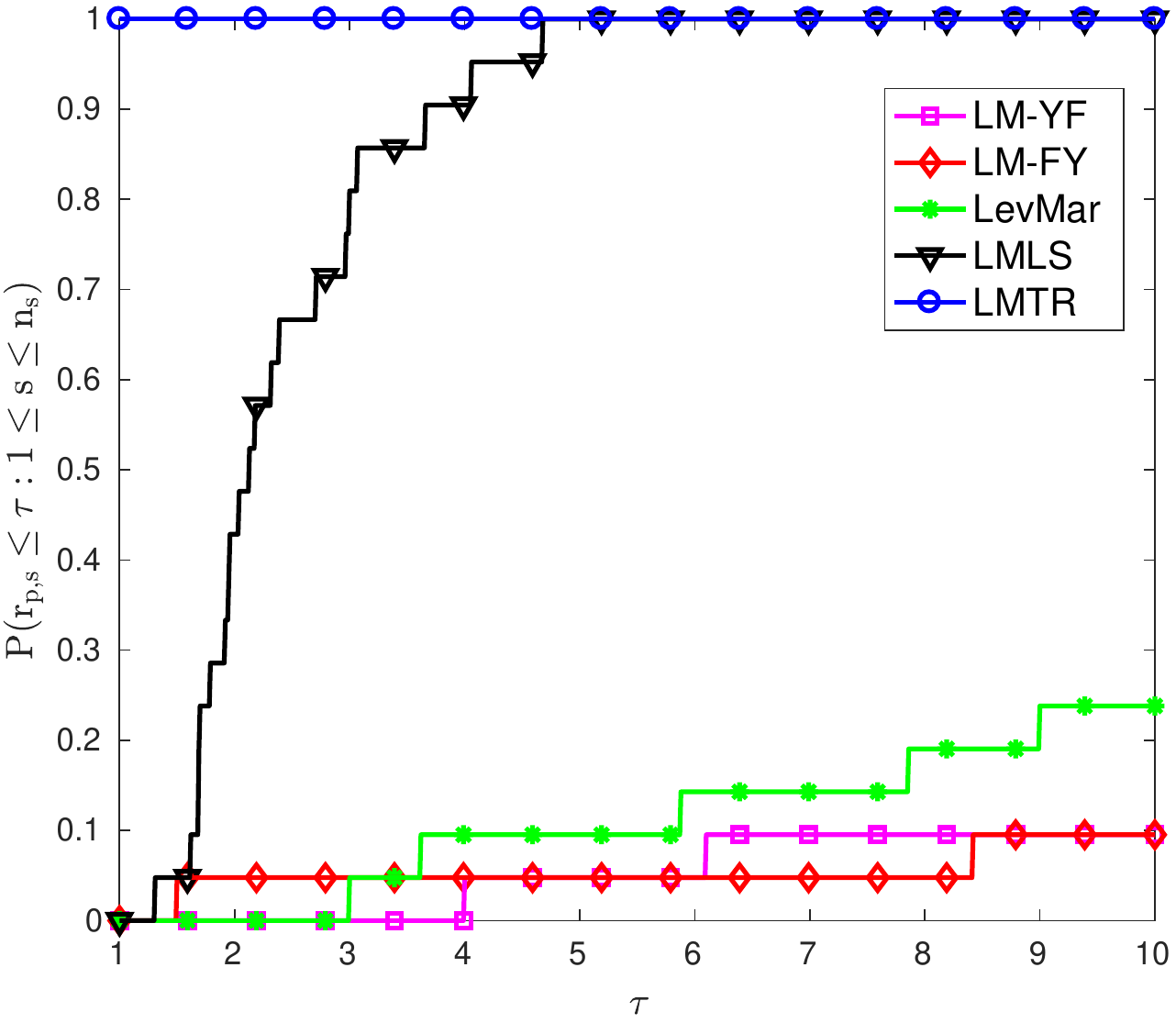}}}%
\end{minipage}\hfill{}%
\begin{minipage}[b][1\totalheight][t]{0.33\columnwidth}%
\subfigure[the number of function evaluations $N_f$]
{\foreignlanguage{english}{\includegraphics[width=7.5cm]{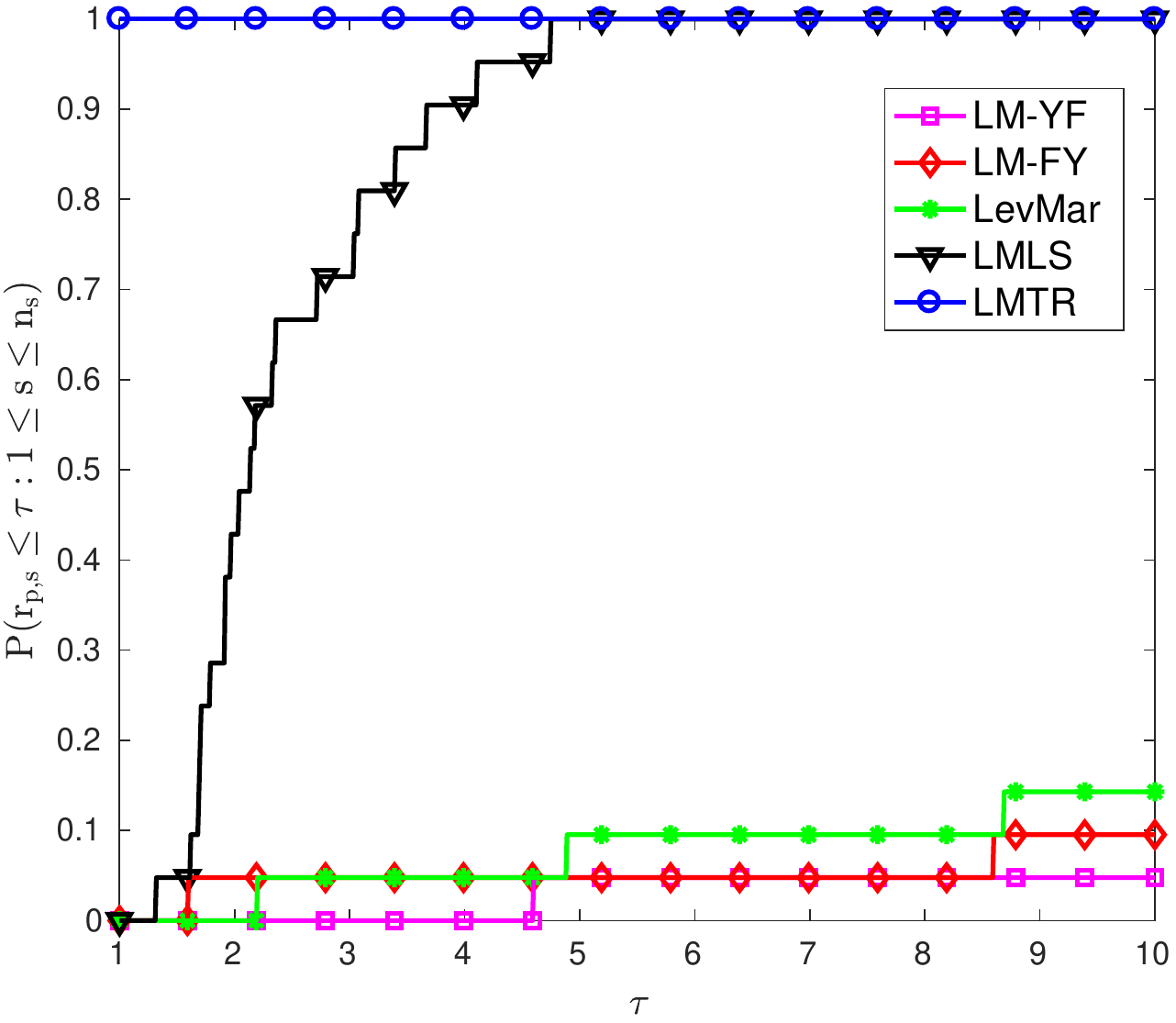}}\label{fig:AnotherFig2-3}}%
\end{minipage}\hspace{2.5cm}

\centering%
\begin{minipage}[b][1\totalheight][t]{0.33\columnwidth}%
\subfigure[the mixed measure $3N_i+N_f$]
{\label{fig:AnotherFig-3-1}\foreignlanguage{english}{\includegraphics[width=7.5cm]{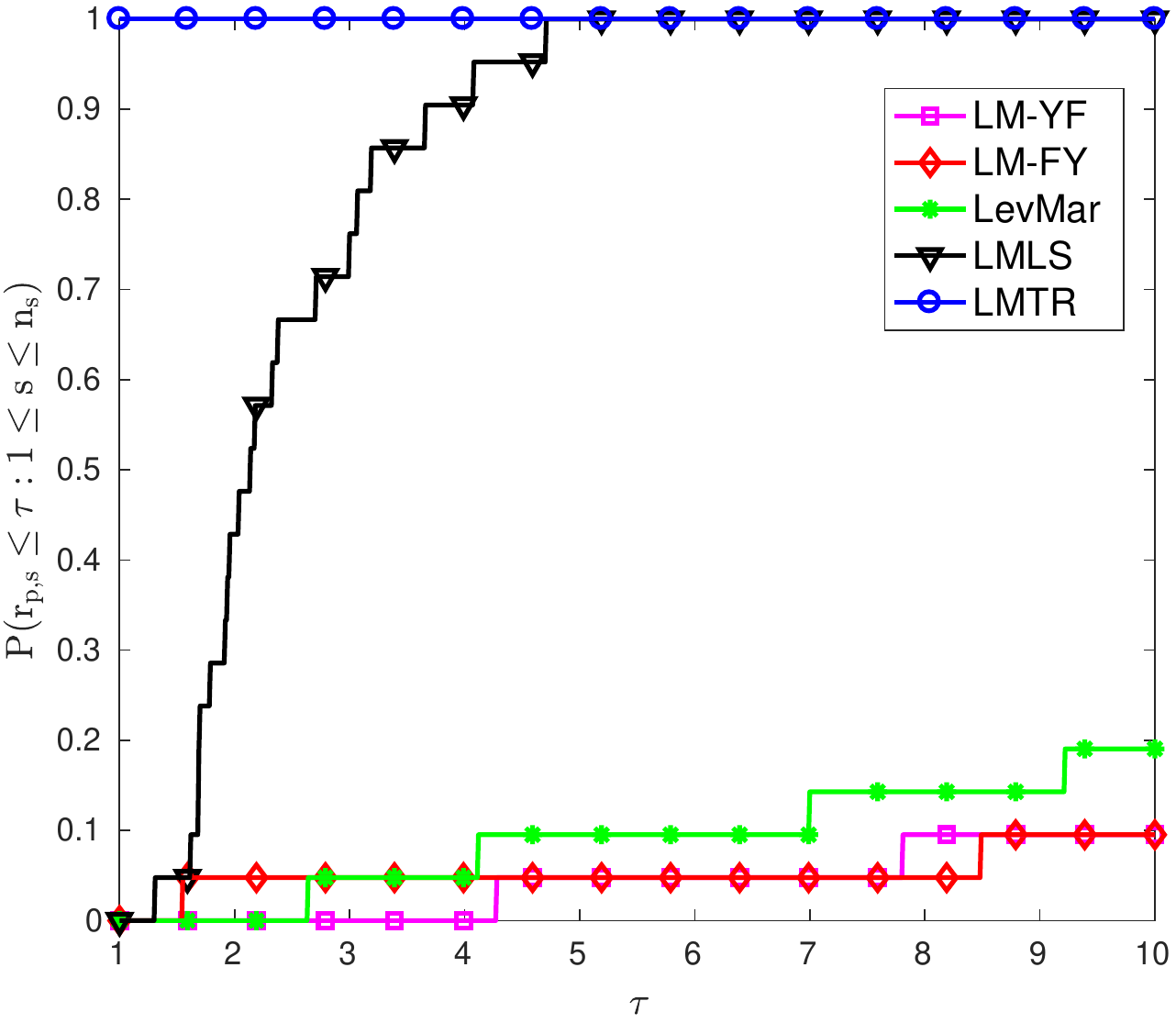}}}%
\end{minipage}\hfill{}%
\begin{minipage}[b][1\totalheight][t]{0.33\columnwidth}%
\subfigure[the running time $T$]
{\foreignlanguage{english}{\includegraphics[width=7.5cm]{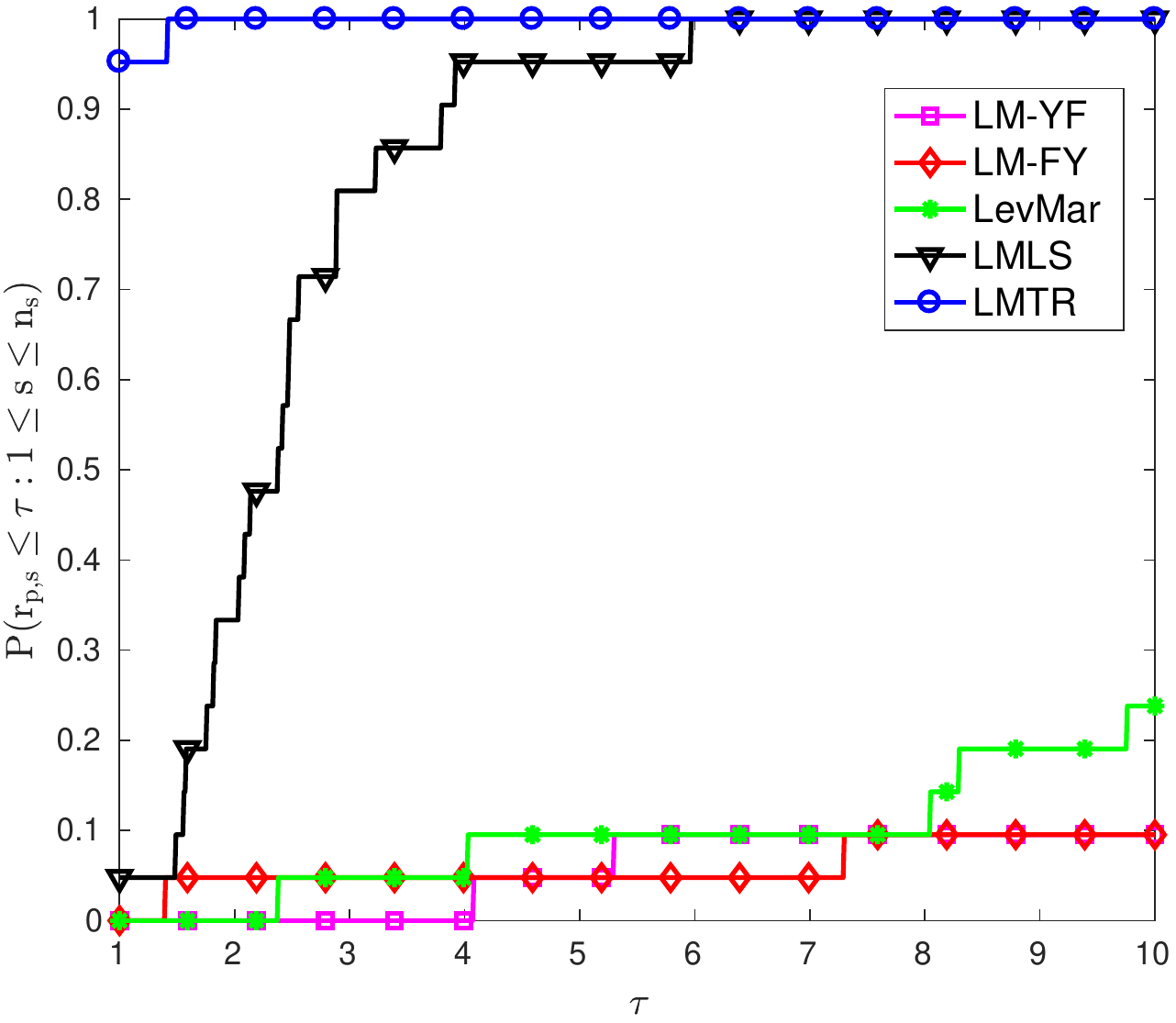}}\label{fig:AnotherFig2-3-1}}%
\end{minipage}\hspace{2.5cm}

\selectlanguage{english}%
\caption{Performance profiles for the number of iterations and the running
time of LM-YF, LM-FY, LevMar, LMLS, and LMTR on a set of 21 biological
models for the mapping (\ref{eq:steadyStateEquation}), where LMLS and LMTR
outperform the others substantially.\label{fig:per_pro_global}}
\selectlanguage{american}%
\end{figure}
%%%%%%%%%%%%

Finally, we conclude this section by displaying the evolution of the merit function
values during run of the considered algorithms. To this end, we illustrate 
the function values $\psi$ versus iterations in Figure~\ref{fig:fan_vs_iter} for the 
mapping~\eqref{eq:steadyStateEquation} with the biological models iBsu1103 and iSB619.
Here, we limit the maximum number of iterations to 1,000. From Figure~\ref{fig:fan_vs_iter},
it can be seen that LMLS and LMTR perform much better than the others; however, the 
best performance is attained by LMTR.

\begin{figure}[H]
\centering%
\subfigure[iBsu1103]{\label{fig:AnotherFig}{\includegraphics[width=.47\textwidth]{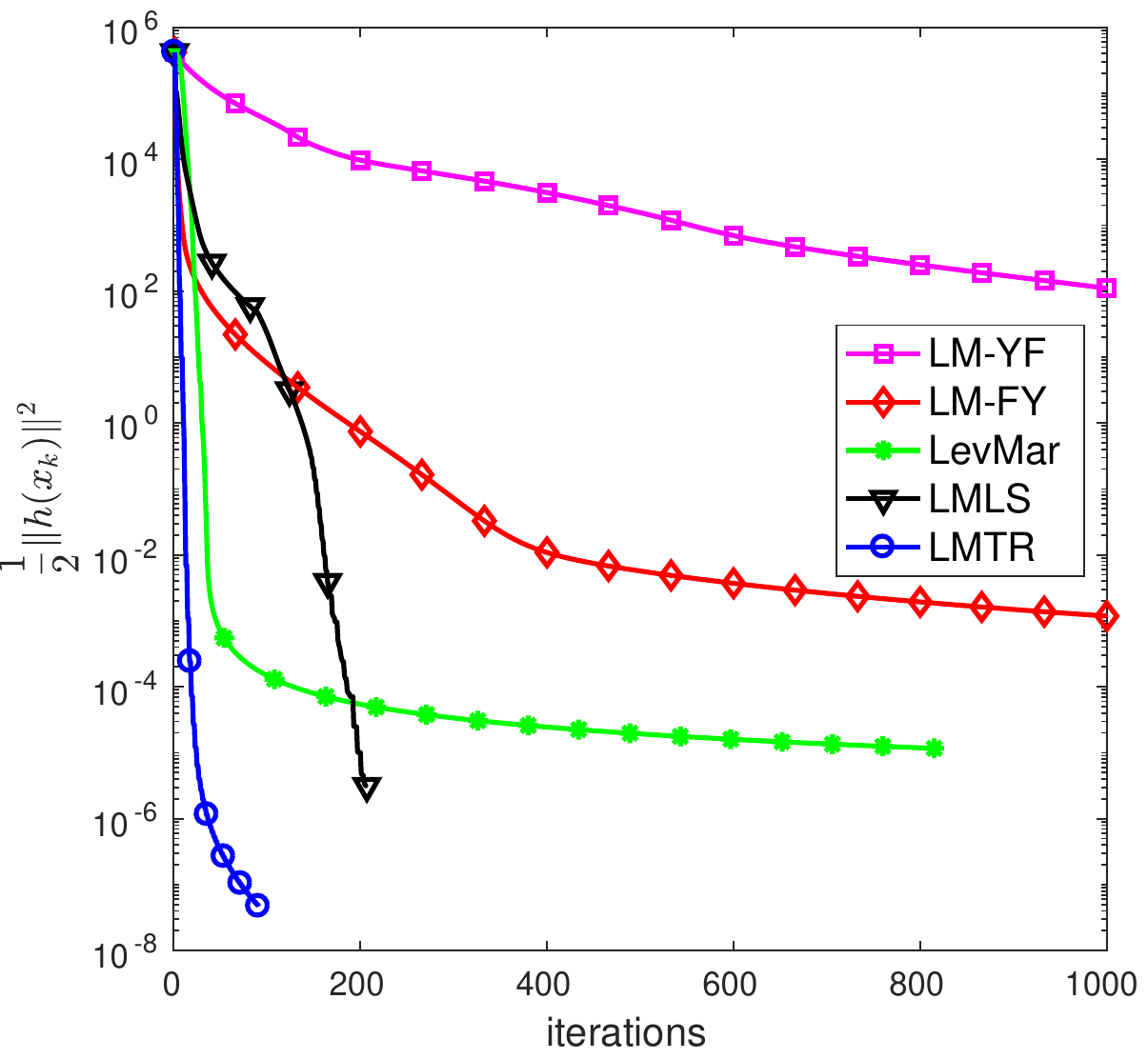}}}\qquad
\subfigure[iSB619]{{\includegraphics[width=.47\textwidth]{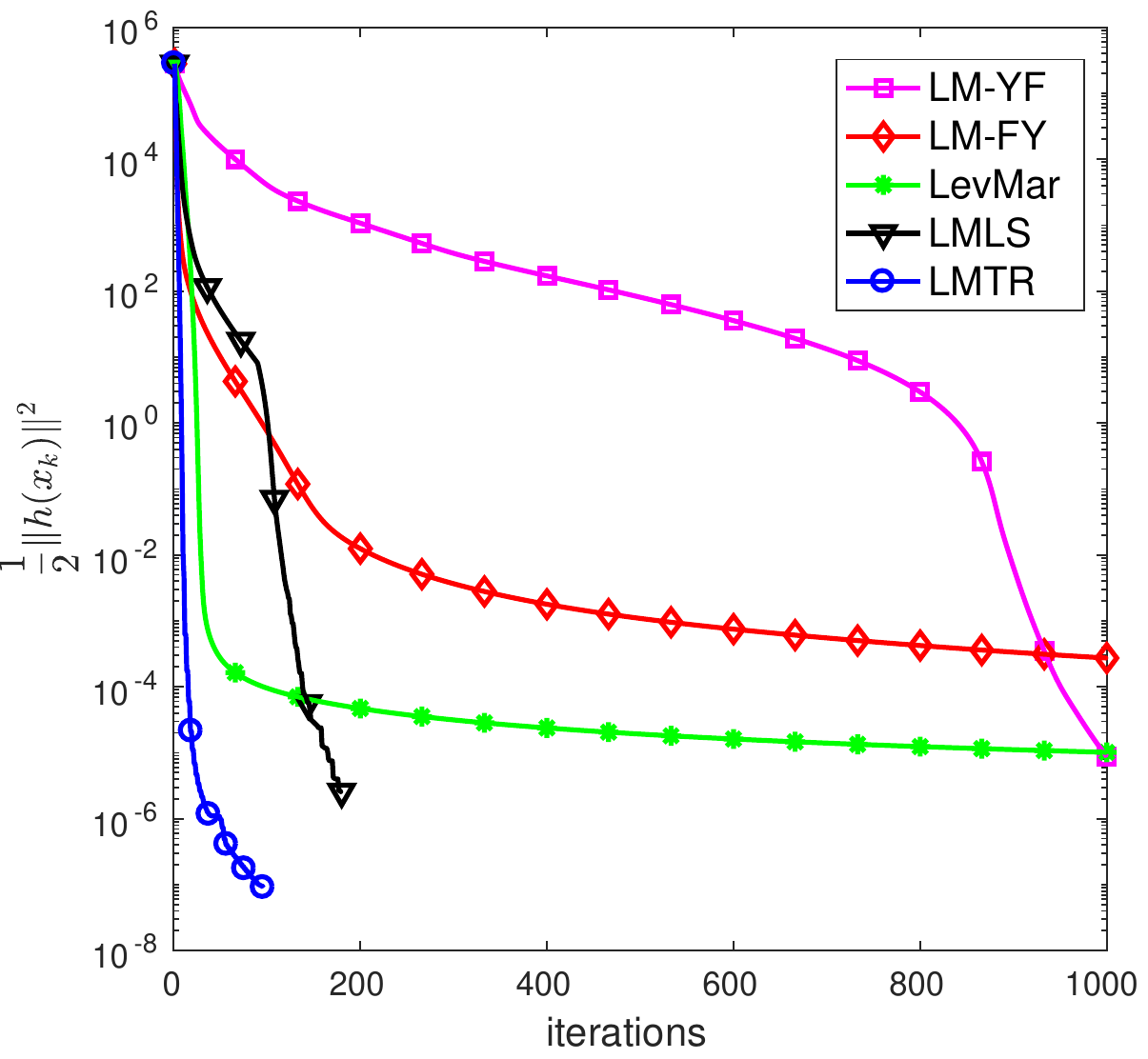}}\label{fig:AnotherFig2}}

\vspace{-3mm}
\caption{Value of the merit function with respect to the number of iterations for LM-YF, LM-FY, 
LevMar, LMLS and LMTR, when applied to the mapping~\eqref{eq:steadyStateEquation} defined by the biological
models iBsu1103 and iSB619. LMLS and LMTR require much less iterations than the others to achieve the accuracy given in \eqref{eq:stopcrit}. \label{fig:fan_vs_iter}}
\end{figure}
%%%%%%%%%%%%

\section{Conclusion and further research }
We have employed two globalisation techniques for Levenberg-Marquardt
methods for finding a zero of H\"{o}lder metrically subregular mappings. First, we 
combined the Levenberg-Marquardt direction with a nonmonotone Armijo-type 
line search. Then, we modified the Levenberg-Marquardt parameter and 
combined the corresponding direction with a nonmonotone trust-region technique. Next, 
we studied the global convergence and the worst-case
global and evaluation complexities or both methods, which are of the order 
$\mathcal{O}(\varepsilon^{-2})$. The worst-case behavior
of the proposed methods, up to a factor, are equivalent to that
of the steepest descent method for unconstrained optimisation, cf.
\cite{CarGT,NesB}, which is not the best-known global complexity for
nonconvex problems, cf. \cite{cartis2011,nesterov2006}; however, practical 
usage of these methods show much better performance than the worse-case complexity, giving scope for future establishement of tigher complexity bounds. Finally, we have studied some special mappings that 
satisfy certain conditions for a stationary point to corresponds to a zero of a mapping, when obtained with the proposed methods.

We also investigate finding zeros of H\"{o}lder metrically subregular mappings that 
appear in modelling of biochemical reaction networks. Our numerical
experiments establish the suitability of the proposed methods for a range of medium- and
large-scale biochemical network problems. Nevertheless, biochemical reaction networks on the order of tens of millions of dimensions already exist~\cite{Magnsdttir2016}, and the projection is for even larger models in the future. Therefore, considerable scope exists for development of accelerated solution methods.

\section*{Acknowledgements}
This work was supported by the U.S. Department of Energy, Offices of Advanced Scientific Computing Research and the Biological and Environmental Research as part of the Scientific Discovery Through Advanced Computing program, grant \#DE-SC0010429, the OPEN program from the Fonds National de la Recherche Luxembourg (FNR/O16/11402054), and the University of Luxembourg (IRP/OptBioSys).

\bibliographystyle{acm}
%\bibliography{lmbiblio}

\end{document}